\documentclass[authoryear, 11pt, a4wide]{amsart}

\usepackage{amsthm,amsmath}
\usepackage{color}
\RequirePackage[colorlinks,citecolor=blue,urlcolor=blue,linkcolor=blue]{hyperref}
\usepackage{amsrefs}

\usepackage[utf8]{inputenc}
\numberwithin{equation}{section}
\usepackage{hyperref,aliascnt}
\usepackage{latexsym,euscript}
\usepackage{amsbsy,amscd,amsfonts,amsmath,amsopn,amssymb,amstext,amsthm,amsxtra,enumitem,bbm}

\usepackage{epsf,epsfig,graphics, color}
\usepackage{ float, xargs, pdfsync}
\usepackage{appendix}

\usepackage{fancybox,ifthen,twoopt}

\newtheorem{theorem}{Theorem}
\newaliascnt{proposition}{theorem}
\newtheorem{proposition}[proposition]{Proposition}
\aliascntresetthe{proposition}

\newaliascnt{lemma}{theorem}
\newtheorem{lemma}[lemma]{Lemma}
\aliascntresetthe{lemma}

\theoremstyle{definition}
\newtheorem{example}{Example}
\newtheorem{remark}{Remark}
\newaliascnt{corollary}{theorem}

\aliascntresetthe{corollary}

\newaliascnt{definition}{theorem}
\newtheorem{definition}[definition]{Definition}
\aliascntresetthe{definition}

\newcounter{hypA'}

\newenvironment{hyp}[1]{%\begin{sf}\refstepcounter{hyp#1}
\begin{enumerate}[label=(\textbf{\sf #1}-\arabic*),resume=hyp#1]\begin{sf}}
{\end{sf}\end{enumerate}}%\end{sf}}

\def\rme{\mathrm{e}}

\def\Xset{\mathsf{X}}
\def\Yset{\mathsf{Y}}

\def\Xsigma{\mathcal{X}}
\def\Ysigma{\mathcal{Y}}

\def\met{\Delta}
\def\rset{\ensuremath{\mathbb{R}}}

\def\zset{\ensuremath{\mathbb{Z}}}

\def\mcf{\mathcal{F}}

\def\eqsp{\;}
\def\PE{\mathbb{E}}

\def\PP{\mathbb{P}}

\def\rmd{\mathrm{d}}
\newcommandx\LogInt[5][1=\theta,4=,5=Y]{\upsilon_{#4}^{#1}\langle {#5}_{#2:#3} \rangle}

\newcommand{\wrt}{with respect to}

\newcommand{\lhs}{left-hand side}

\newcommand{\as}{\mbox{-a.s.}}

\newcommand{\ie}{i.e.}

\newcommandx{\aslim}[1]{\ensuremath{\stackrel{#1-\text{a.s.}}{\longrightarrow}}}

\newcommandx\sequence[3][2=n,3=\zset]{\ensuremath{\left(#1_{#2}\right)_{#2 \in #3}}}
\newcommand\nsequence[2]{\ensuremath{\left(#1\right)_{#2}}}
\newcommandx\dsequence[4][3=n,4=\zset]{\ensuremath{\left( (#1_{#3},#2_{#3})\right)_{#3 \in #4}}}

\newcommand{\CPE}[3][]
{\ifthenelse{\equal{#1}{}}%
{\mathbb{E}\left[\left. #2 \, \right| #3 \right]}
{\mathbb{E}_{#1}\left[\left. #2 \, \right| #3 \right]}
}
\newcommand{\CPEu}[4][]
{\ifthenelse{\equal{#1}{}}%
{\mathbb{E}\left[\left. #2 \, \right| #3 \right]}
{\mathbb{E}^{#1}_{#2}\left[\left. #3 \, \right| #4 \right]}
}

\newcommand{\CPEv}[3][]
{\ifthenelse{\equal{#1}{}}%
{\mathbb{E}_\star\left[\left. #2 \, \right| #3 \right]}
{\mathbb{E}_\star_{#1}\left[\left. #2 \, \right| #3 \right]}
}

\newcommand{\CPP}[3][]
{\ifthenelse{\equal{#1}{}}%
{\mathbb{P}\left[\left. #2 \, \right| #3 \right]}
{\mathbb{P}_{#1}\left[\left. #2 \, \right| #3 \right]}
}

\newcommand{\CPPu}[3][]
{\ifthenelse{\equal{#1}{}}%
{\mathbb{P}\left[\left. #2 \, \right| #3 \right]}
{\mathbb{P}^{#1}\left[\left. #2 \, \right| #3 \right]}
}

\newcommandx{\chunk}[3]%
{\ensuremath{#1}_{#2:#3}}

\def\1{\mathbbm{1}}

\newcommandx\proj[2][1=,2=]{
\ifthenelse{\equal{#1}{}}
{\operatorname{X}}
{\operatorname{X}_{#1:#2}}i
}

\newcommand{\eqdef}{:=}
\newcommand{\set}[2]{\left\{#1\, : \, #2\right\}}

\newcommandx\lkdM[3][1=,3=]{
\ifthenelse{\equal{#2}{}}
{ \mathsf{L}_{#1}^{#3}}
{ \mathsf{L}_{#1}^{#3}\langle #2\rangle}
}
\newcommandx\lkdMStat[3][1=,3=]{
\ifthenelse{\equal{#2}{}}
{ \bar{\mathsf{L}}_{#1}^{#3}}
{ \bar{\mathsf{L}}_{#1}^{#3}\langle #2 \rangle}
}

\newcommandx\lkd[3][1=,3=]{
\ifthenelse{\equal{#2}{}}
{ \ell_{#1}^{#3}}
{ \ell_{#1}^{#3}\langle #2\rangle}
}
\newcommandx\lkdStat[3][1=,3=]{
\ifthenelse{\equal{#2}{}}
{ \bar \ell_{#1}^{#3}}
{ \bar \ell_{#1}^{#3}\langle #2 \rangle}
}

\newcommand{\mlY}[1]{\hat\theta_{#1}}
\newcommand{\argmax}{\mathop{\mathrm{argmax}}}

\newcommandx{\normLip}[2][1=]{\mathrm{Lip}(#2;#1)}
\newcommandx{\wass}[2][1]{\lVert #2\rVert_{#1}}

\newcommandx{\wasser}[3][1=]{\mathcal{W}_{#1}\left(#2,#3\right)}
\newcommandx{\proho}[3][1=]{\mathcal{P}_{#1}\left(#2,#3\right)}

\newcommandx{\dobru}[2][1=]{\dobrush_{#1}\left( #2\right)}
\newcommand{\dobrush}{\Delta}

\newcommandx{\Pcan}[2][1=,2=]{\mathbb{P}_{#1}^{#2}}
\newcommandx{\Ecan}[2][1=,2=]{\mathbb{E}_{#1}^{#2}}
\newcommand{\Xinit}{\xi}

\newcommandx\cesp[4][1=,2=]{\ensuremath{{\mathbb E}_{#1}^{#2}\left[ \left. #3 \right| #4 \right]}}

\newcommandx{\f}[2][1=\theta]{\psi^{#1}\langle #2 \rangle}
\newcommandx{\F}[2][1=\theta]{\Psi^{#1}\langle #2 \rangle}
\newcommand{\thv}{{\theta_\star}}
\newcommandx{\kap}[3][1=\theta]{
\ifthenelse{\equal{#3}{}}
{\kappa^{#1}\langle #2\rangle }
{\kappa^{#1}\langle #2\rangle \left(#3\right)}
}

\def\lnp{\ln^{+}}

\def\zsetp{\zset_{+}}
\def\zsetn{\zset_{-}}
\def\zsetpnz{\zset_{+}^*}
\def\zsetnnz{\zset_{-}^*}

\def\rsetp{\rset_{+}}
\def\rsetn{\rset_{-}}

\def\rsetpnz{\rset_{+}^*}
\def\rsetnnz{\rset_{-}^*}

\newcommandx{\probdoeblin}[3][1=]{\mu^{#1}_{#2}\langle #3 \rangle}
\newcommand{\Pblock}[2][]
{\ifthenelse{\equal{#1}{}}{\boldsymbol{\operatorname{L}}\langle#2\rangle}{\boldsymbol{\operatorname{L}}^{#1}\langle#2\rangle}
}

\newcommand{\ConPblock}[3][]
{\ifthenelse{\equal{#1}{}}{\boldsymbol{\operatorname{L}}\langle#2|#3\rangle}{\boldsymbol{\operatorname{L}}^{#1}\langle#2|#3\rangle}
}

\newcommand{\pblock}[2][]
{\ifthenelse{\equal{#1}{}}{\mathbf{\ell}\langle#2\rangle}{\mathbf{\ell}^{#1}\langle #2\rangle}
}

\def\simplex{\mathsf{P}}

\newcommandx{\limlike}[4][1=\theta, 2=\theta_\star]{p^{#1,#2}\left( #3\mid #4 \right)}

\def\Xmet{\rmd_\Xset}

\begin{document}

%%\issueinfo{5}{2012}
%%\pagespan{1}{8}
%%\received{October 31, 2011}
%%\DOI{123-xxxxxxxxxxx}
%%

%\begin{frontmatter}




\title[Partially observed Markov chains]{The maximizing set of the asymptotic normalized log-likelihood for
  partially observed Markov chains}

%\runtitle{Partially observed Markov chains}

%\begin{aug}
%\author{\fnms{Randal} \snm{Douc}\thanksref{m1}\ead[label=e1]{randal.douc@telecom-sudparis.eu}},
%\author{\fnms{Fran\c{c}ois} \snm{Roueff}\thanksref{m2}\ead[label=e2]{roueff@telecom-paristech.fr}}
%\and
%%%\author{\fnms{Tepmony} \snm{Sim}\thanksref{t1,m2}
%\author{\fnms{Tepmony} \snm{Sim}\thanksref{m2}
%\ead[label=e3]{sim@telecom-paristech.fr}
%%\ead[label=u1,url]{http://www....}
%}
%
%%%\thankstext{t1}{Supported by Institut Carnot}
%%\thankstext{t2}{First supporter of the project}
%%\thankstext{t3}{Second supporter of the project}
%\runauthor{R. Douc et al.}
%
%\affiliation{Telecom Sudparis\thanksmark{m1}
%and Telecom Paristech\thanksmark{m2}}
%
% \address{Department  CITI\\
% CNRS UMR 5157\\ 
% Telecom Sudparis\\ 
% Evry\\
% France\\
% \printead{e1}\smallskip}
%
% \address{Institut Mines-Telecom\\
% Telecom Paristech\\ 
% CNRS LTCI\\
% Paris\\ 
% France\\
% \printead{e2}\\
% \phantom{E-mail:\ }\printead*{e3}\\
% %\printead{u1}
% }
%\end{aug}

%\author[tsp]{R. Douc}
%\address[tsp]{Department  CITI, CNRS UMR 5157, Telecom Sudparis, Evry. France.}
%\ead{randal.douc@telecom-sudparis.eu}
%
%\author[tpt]{F. Roueff}
%\address[tpt]{Telecom Paristech, Institut Telecom, CNRS LTCI, Paris. France.}
%\ead{roueff@telecom-paristech.fr}
%
%\author[tpt]{T. Sim}
%\ead{sim@telecom-paristech.fr}

\author[R. Douc et al.]{Randal Douc$^1$}
\address{$^1$Department  CITI\\
CNRS UMR 5157\\
Telecom Sudparis\\
Evry\\
France}
\email[R.~Douc]{\href{mailto:randal.douc@telecom-sudparis.eu}{randal.douc@telecom-sudparis.eu}}

\author[ ]{Fran\c{c}ois Roueff$^2$}
\address{$^2$Institut Mines-Telecom\\
Telecom Paristech\\
CNRS LTCI\\
 Paris\\
 France}
\email[F.~Roueff]{\href{mailto:roueff@telecom-paristech.fr}
{roueff@telecom-paristech.fr}}

\author[ ]{Tepmony Sim$^2$}
%\address[T. Sim]{Institut Mines-Telecom, Telecom Paristech, CNRS LTCI, Paris. France.}
\email[T.~Sim]{\href{mailto:sim@telecom-paristech.fr}{sim@telecom-paristech.fr}}

\subjclass[2010]{Primary 60J05, 62F12; Secondary 62M05, 62M10}

\keywords{consistency, ergodicity, hidden Markov models, maximum likelihood, observation-driven models, time series of counts}

\begin{abstract}
  This paper deals with a parametrized family of partially observed bivariate
  Markov chains. We establish that, under very mild assumptions, the limit of
  the normalized log-likelihood function is maximized when the parameters belong to the
  equivalence class of the true parameter, which is a key feature for obtaining
  the consistency of the maximum likelihood estimators (MLEs) in well-specified
  models. This result is obtained in the general framework of partially
  dominated models. We examine two specific cases of interest, namely, hidden
  Markov models (HMMs) and observation-driven time series models. In contrast with previous
  approaches, the identifiability is addressed by relying on the uniqueness of the
  invariant distribution of the Markov chain associated to the complete data,
  regardless its rate of convergence to the equilibrium.
\end{abstract}

%\begin{keyword}[class=MSC]
%\kwd[Primary ]{60J05}
%\kwd{62F12}
%\kwd[; secondary ]{62M05}
%\kwd{62M10}
%\end{keyword}
%
%\begin{keyword}
%\kwd{consistency}
%\kwd{ergodicity}
%\kwd{hidden Markov models}
%\kwd{maximum likelihood}
%\kwd{observation-driven models}
%\kwd{time series of counts}
%\end{keyword}
%
%\end{frontmatter}

\maketitle

\sloppy

\section{Introduction}

Maximum likelihood estimation is a widespread method for identifying a
parametric model of a time series from a sample of observations.  Under a
well-specified model assumption, it is of prime interest to show the
consistency of the estimator, that is, its convergence to the true parameter,
say $\thv$, as the sample size goes to infinity. The proof generally involves
two important steps: 1) the maximum likelihood estimator (MLE) converges to the
maximizing set $\Theta_\star$ of the asymptotic normalized log-likelihood, and 2)
the maximizing set indeed reduces to the true parameter. The second step is
usually referred to as solving the \emph{identifiability} problem but it can
actually be split in two sub-problems: 2.1) show that any parameter in
$\Theta_\star$ yields the same distribution for the observations as for the
true parameter, and 2.2) show that for a sufficiently large sample size, the set
of such parameters reduces to $\thv$. Problem 2.2 can be difficult to solve,
see \cite{allman:matias:rhodes:2009,gassiat:cleymen:robin:2013} and the
references therein for recent advances in the case of hidden Markov models
(HMMs).  Nevertheless, Problem~2.1 can be solved independently, and with
Step 1 above, this directly yields that the MLE is
consistent in a weakened sense, namely, that the estimated parameter converges
to the set of all the parameters associated to the same distribution as the one of
the observed sample. This consistency result is referred to as
\emph{equivalence-class consistency}, as introduced by \cite{leroux:1992}.  In
this contribution, our goal is to provide a general approach to solve
Problem~2.1 in the general framework of partially observed Markov models. These include
many classes of models of interest, see for instance \cite{pieczynski:2003} or
\cite{ephraim:mark:2012}. The novel aspect of this work is that the result
mainly relies on the uniqueness of the invariant distribution of the Markov chain
associated to the complete data, regardless its rate of convergence to the
equilibrium.  We then detail how this approach applies in the context of two
important subclasses of partially observed Markov models, namely, the class of
HMMs and the class of observation-driven time series models.

In the context of HMMs, the consistency of the MLE is of primary importance,
either as a subject of study (see
\cite{leroux:1992,douc:moulines:ryden:2004,douc:moulines:olsson:vanhandel:2011})
or as a basic assumption (see
\cite{bickel:ritov:ryden:1998,jensen:petersen:1999}). The characterization of
the maximizing set $\Theta_\star$ of the asymptotic log-likelihood (and thus
the equivalence-class consistency of the MLE) remains a delicate question for
HMMs. As an illustration, we consider the following example. In this
example and throughout the paper we denote by $\rsetp=[0,\infty)$,
$\rsetn=(-\infty,0]$,  $\rsetpnz=(0,\infty)$ and $\rsetnnz=(-\infty,0)$, the sets of nonnegative,
nonpositive, (strictly) positive and  (strictly) negative real numbers, respectively. Similarly, we
use the notation $\zsetp$, $\zsetn$, $\zsetpnz$  and $\zsetnnz$ for the corresponding
subsets of integers. Also, $a^+=\max(a,0)$ denotes the nonnegative part of $a$.
\begin{example}\label{exple:hmm}
  Set $\Xset=\rsetp$, $\Xsigma={\mathcal B}(\rsetp)$, $\Yset=\rset$ and
  $\Ysigma={\mathcal B}(\rset)$ and define an HMM on $\Xset
  \times \Yset$ by the following recursions:
\begin{align}\label{eq:def:exemple:hmm}
\begin{split}
&X_{k}=\left(X_{k-1}+U_k-m\right)^+\eqsp, \\
&Y_ {k}= a X_{k}+V_k\eqsp,
\end{split}
\end{align}
where $(m, a) \in \rsetpnz \times \rset$, and the sequence
$\dsequence{U}{V}[k][\zsetp]$ is independent and identically distributed (i.i.d.) and is independent from $X_0$.
This Markov model $\sequence{X}[k][\zsetp]$ was proposed by  \cite{tuominen:tweedie:1994} and further considered by
\cite{jarner:roberts:2002} as an example of polynomially ergodic Markov chain,
under specific assumptions made on $U_k$'s. Namely, if $U_k$'s are centered
and $\PE[\rme^{\lambda U_k^+}]=\infty$ for any $\lambda>0$, it can be shown that the chain
$\sequence{X}[k][\zsetp]$ is not geometrically ergodic (see \autoref{lem:exemple:hmm} below). In such a situation, the exponential
separation of measures condition introduced in
\cite{douc:moulines:olsson:vanhandel:2011} seems difficult to check. We will
show, nevertheless, in \autoref{prop:consistanceYexample}, that under some mild
conditions the chain $\sequence{X}[k][\zsetp]$ is ergodic and the equivalence-class consistency
holds.
\end{example}

Observation-driven time series models were introduced by \cite{cox:1981} and later
considered, among others, by \cite{streett:2000, davis:dunsmuir:streett:2003, fokianos:rahbek:tjostheim:2009, neuman:2011, doukhan:fokianos:tjostheim:2012} and
\cite{dou:kou:mou:2013}. The celebrated GARCH(1,1) model introduced by
\cite{bollerslev:1986} is an observation-driven model as well as most of the
models derived from this one, see \cite{bollerslev08-glossary} for a list of
some of them. This class of models has the nice feature that the (conditional)
likelihood function and its derivatives are easy to compute. The consistency of the MLE can however be cumbersome and is often derived using computations specific to
the studied model. When the observed variable is discrete, general consistency
results have been obtained only recently in \cite{davis:liu:2012} or \cite{dou:kou:mou:2013} (see also in
\cite{henderson:matteson:woodard:2011} for the existence of stationary and ergodic solutions to some observation-driven time series models). However, in
these contributions, the way of proving that the maximizing set $\Theta_\star$
reduces to $\{\thv\}$ requires checking
specific conditions in each given example and seems difficult to assert in a
more general context, for instance when the distribution of the observations
given the hidden variable also depends on an unknown parameter. Let us
describe two such examples.
The first one (\autoref{example:Nbigarch:defi}) was introduced in \cite{zhu:2011}. Up to our knowledge the
consistency of the MLE has not been treated for this model.

\begin{example}\label{example:Nbigarch:defi}
The negative binomial integer-valued GARCH (NBIN-GARCH$(1,1)$) model is defined
by:
\begin{align}\label{eq:def:nbingarch}
\begin{split}
&X_{k+1}=\omega+a X_k+bY_{k}\eqsp, \\
&Y_ {k+1}|\chunk{X}{0}{k+1},\chunk{Y}{0}{k}\sim \mathcal{NB} \left(r,\frac{X_{k+1}}{1+X_{k+1}}\right)\eqsp,
\end{split}
\end{align}
where $X_k$ takes values in $\Xset=\rsetp$,  $Y_k$ takes values in $\zsetp$ and
$\theta=(\omega, a, b,r)\in(\rsetpnz)^4$ is an unknown parameter. In \eqref{eq:def:nbingarch},
$\mathcal{NB}(r, p)$ denotes the negative binomial distribution with parameters
$r>0$ and $p\in (0,1)$, whose probability function is $\frac{\Gamma(k+r)}{k!\Gamma(r)}
p^{r}(1-p)^k$ for all $k\in\zsetp$, where $\Gamma$ stands
for the Gamma function.
\end{example}
The second example, \autoref{example:nmgarch:defi}, proposed by \cite{haas2004mixed} and \cite{alexander2006normal}, is a natural extension of GARCH processes, where the
usual Gaussian conditional distribution of the observations given the hidden
volatility variable is replaced by a mixture of Gaussian distributions given a
hidden vector volatility variable. Up to our knowledge, the usual
consistency proof of the MLE for the GARCH cannot be directly adapted to this
model.

\begin{example}%[Normal Mixture GARCH Model]
\label{example:nmgarch:defi}
The normal mixture GARCH (NM$(d)$-GARCH$(1,1)$) model is defined by:
\begin{align}\label{eq:mmdm:def}
\begin{split}
&\mathbf{X}_{k+1}=\boldsymbol{\omega}+\mathbf{A}\mathbf{X}_k+Y_{k}^2\mathbf{b} \eqsp,\\
&Y_ {k+1}|\chunk{\mathbf{X}}{0}{k+1},\chunk{Y}{0}{k}\sim G^\theta(\mathbf{X}_{k+1};\cdot) \eqsp, \\
&G^\theta(\mathbf{x};\rmd y)=
\left(\sum_{\ell=1}^{d}{\gamma_\ell\frac{\rme^{-y^2/{2x_{\ell}}}}{(2\pi
      x_{\ell})^{1/2}}}\right)\rmd y\eqsp,\; \mathbf{x}=(x_i)_{1\leq i\leq d}\in(\rsetpnz)^d,\;y\in\rset\eqsp,
\end{split}
\end{align}
where $d$ is a positive integer; $\mathbf{X}_k=[X_{1,k} \ldots X_{d,k}]^T$ takes values in  $\Xset=\rsetp^d$;
$\boldsymbol{\gamma}=[\gamma_1\ldots\gamma_d]^T$ is a $d$-dimensional vector of
mixture coefficients belonging to the $d$-dimensional simplex
$\simplex_d=\{\boldsymbol{\gamma}\in\rsetp^d:\; \sum_{\ell=1}^{d}{\gamma_\ell}=1\}$;
$\boldsymbol{\omega}$, $\mathbf{b}$ are $d$-dimensional
vector parameters with positive and nonnegative entries, respectively;
$\mathbf{A}$ is a $d\times d$ matrix parameter with nonnegative entries; and
$\theta=\left(\boldsymbol{\gamma}, \boldsymbol{\omega}, \mathbf{A},
  \mathbf{b}\right)$.
\end{example}

The paper is organized as follows. \autoref{sec:identifiability} is dedicated to the main result (\autoref{thm:mainresult}) which
shows that the argmax of the limiting criterion reduces to the equivalence
class of the true parameter, as defined in \cite{leroux:1992}. The general setting is introduced in \autoref{subsec:general:setting}. The theorem is stated and proved in \autoref{subsec:main}. In \autoref{subsec:backward}, we focus on the kernel involved in the
assumptions, and explain how it can be obtained explicitly. Our general assumptions are then shown to hold for two important classes of partially observed Markov models:
\begin{enumerate}[label=-]
\item First, the HMMs  described in \autoref{sec:hmm}, for
which the equivalence-class consistency of the MLE is derived under simplified
assumptions. The polynomially ergodic HMM of \autoref{exple:hmm} is treated as
an application of this result.
\item Second, the observation-driven time series models described in
  \autoref{sec:observ-driv-model}. The obtained results apply to the models of
  \autoref{example:Nbigarch:defi} and \autoref{example:nmgarch:defi}, where the
  generating process of the observations may also depend on the parameter.
\end{enumerate}
The technical proofs are gathered in \autoref{sec:postponed-proofs}.

\section{A general approach to identifiability}
\label{sec:identifiability}
 \subsection{General setting and notation:~partially dominated and partially observed Markov models}
\label{subsec:general:setting}
 Let $(\Xset,\Xsigma)$ and $(\Yset,\Ysigma)$ be two Borel spaces, that is,
 measurable spaces that are isomorphic to a Borel subset of $[0,1]$ and let
 $\Theta$ be a set of parameters. Consider a statistical model determined by a
 class of Markov kernels $\nsequence{K^\theta}{\theta\in\Theta}$ on
 $(\Xset\times \Yset) \times (\Xsigma \otimes \Ysigma)$.  Throughout the paper,
 we denote by $\mathbb{P}_{\xi}^\theta$ the probability (and by
 $\mathbb{E}_{\xi}^\theta$ the corresponding expectation) induced on
 $(\Xset\times\Yset)^{\zsetp}$ by a Markov chain $\dsequence{X}{Y}[k][\zsetp]$ with
 transition kernel $K^\theta$ and initial distribution $\xi$ on
 $\Xset\times\Yset$.  In the case where $\xi$ is a Dirac mass at $(x,y)$ we
 will simply write $\mathbb{P}_{(x,y)}^\theta$.

 For partially observed Markov chains, that is, when only a sample $\chunk Y 1
 n\eqdef(Y_1,\dots,Y_n)\in\Yset^n$ of the second component is observed, it is
 convenient to write $K^\theta$ as
 \begin{align}
 K^\theta((x, y); \rmd x'\rmd y')=Q^\theta((x, y); \rmd x')G^\theta((x, y, x'); \rmd y') \eqsp, \label{eq:def:gen:XY:theta}
\end{align}
where $Q^\theta$ and $G^\theta$ are probability kernels on
$(\Xset\times\Yset)\times\Xsigma$ and on
$(\Xset\times\Yset\times\Xset)\times\Ysigma$, respectively.

We now consider the following general setting.

 \begin{definition}\label{def:partially-dom}
   We say that the Markov model $\nsequence{K^\theta}{\theta\in\Theta}$ of the form~(\ref{eq:def:gen:XY:theta}) is
   partially dominated if there exists a $\sigma$-finite measure $\nu$ on
   $\Yset$ such that for all $(x, y), (x', y')\in \Xset\times\Yset$,
\begin{equation}\label{eq:product-form-kernelK}
G^\theta((x, y, x'); \rmd y')=g^\theta((x,y,x'); y')\nu(\rmd y')\eqsp,
\end{equation}
where the conditional density function $g^\theta$ moreover satisfies
\begin{equation}\label{eq:g-evrywhere-strictly-positive}
g^\theta((x,y,x'); y') > 0,\quad\text{for all}\quad(x, y), (x', y')\in
\Xset\times\Yset \;.
\end{equation}
 \end{definition}
\noindent It follows from~(\ref{eq:product-form-kernelK}) that,  for all
$(x, y)\in\Xset\times\Yset$, $A\in \Xsigma$ and $B\in \Ysigma$,
$$
K^\theta\left((x, y);A\times B\right)=\int_B \kap{y, y'}{x; A}\nu(\rmd
y')\;,
$$
where, for all $y,y' \in \Yset$, $\kap{y,y'}{}$ is a kernel defined on
$(\Xset,\Xsigma)$ by
\begin{equation} \label{eq:kappa}
\kap{y, y'}{x;\rmd x'} \eqdef
 Q^\theta((x,y);\rmd x')g^\theta((x,y,x');y') \; .
\end{equation}
\begin{remark} \label{rem:kappa-is-positive} Note that, in general,
  the kernel $\kappa^{\theta}\langle y,y'\rangle$ is unnormalized  since
  $\kap{y,y'}{x; \Xset}$ may be different from one. Moreover, we have for all
  $(x,y,y') \in \Xset \times \Yset \times \Yset$,
\begin{equation} \label{eq:kappa-is-positive}
\kap{y,y'}{x; \Xset}= \int_{\Xset} Q^\theta((x, y);\rmd
    x')g^\theta((x,y,x');y') >0 \; ,
\end{equation}
where the positiveness follows from the fact that $Q^\theta((x, y);\cdot)$ is a probability on $(\Xset, \Xsigma)$ and Condition~(\ref{eq:g-evrywhere-strictly-positive}).
\end{remark}

In well-specified models, it is assumed that the observations $\chunk{Y}{1}{n}$
are generated from a process $\dsequence{X}{Y}[k][\zsetp]$, which follows the
distribution $\PP^{\thv}_{\xi_\star}$ associated to an unknown parameter
$\thv\in\Theta$ and an unknown initial distribution $\xi_\star$ (usually,
$\xi_\star$ is such that, under $\PP^{\thv}_{\xi_\star}$,
$\sequence{Y}[k][\zsetp]$
is a stationary sequence).  To form a consistent estimate of $\thv$ on the
basis of the observations $\chunk Y1n$ only, \ie, without access to the
hidden process $\sequence{X}[k][\zsetp]$, we define the maximum likelihood estimator
(MLE) $\mlY{\Xinit,n}$ by
$$
\mlY{\Xinit,n} \in \argmax_{\theta \in \Theta} L_{\xi,n}(\theta)
\eqsp,
$$
where $L_{\xi,n}(\theta)$ is the (conditional) log-likelihood function of the
observations under parameter $\theta$ with some arbitrary initial distribution
$\xi$ on $\Xset\times\Yset$, that is,
\begin{align*}
L_{\xi,n}(\theta)&\eqdef
\ln \int \prod_{k=1}^n
Q^\theta((x_{k-1},y_{k-1});\rmd x_k)\,g^{\theta}((x_{k-1},y_{k-1}, x_k);y_k)
\xi (\rmd x_0\rmd y_0) \\
&= \ln\int\kap{y_0,y_1}{}\kap{y_1,y_2}{}\dots\kap{y_{n-1},y_n}{x_0;\Xset}\;
\xi (\rmd x_0\rmd y_0)\;.
\end{align*}
This corresponds to the log of the conditional density of $\chunk Y1n$ given
$(X_0,Y_0)$ with the latter integrated according to $\xi$. In practice $\xi$ is
often taken as a Dirac mass at $(x,y)$ with $x$ arbitrarily chosen and $y$
equal to the observation $Y_0$ when it is available.
In this context, a classical way  (see for
example \cite{leroux:1992})
to prove the consistency of a maximum-likelihood-type
estimator $\mlY{\Xinit,n}$ may be decomposed in the following steps. The first step is to show that
$\mlY{\Xinit,n}$ is, with probability tending to one, in a neighborhood of the
set
\begin{equation}\label{eq:ThetaStar-def}
\Theta_\star \eqdef \argmax_{\theta\in\Theta}\tilde{\PE}^{\thv}\left[\ln p^{\theta,\thv}(Y_1|\chunk{Y}{-\infty}{0})\right]\eqsp.
\end{equation}
This formula involves two quantities that have not yet been defined
since they may require additional assumptions:~first, the expectation
$\tilde{\PE}^{\theta}$, which corresponds to the distribution
$\tilde\PP^{\theta}$ of a sequence $\sequence{Y}[k]$ in accordance with the
kernel $K^\theta$, and second, the density $p^{\theta,\thv}(\cdot|\cdot)$, which shows up when
taking the limit, under $\tilde\PP^{\thv}$, of the $\tilde\PP^{\theta}$-conditional density of $Y_1$ given
its $m$-order past, as $m$ goes to infinity. In many cases, such quantities
appear naturally because the model is ergodic and the normalized
log-likelihood $n^{-1}L_{\xi,n}(\theta)$ can be approximated by
$$
\frac1n \sum_{k=1}^n \ln p^{\theta,\thv}(Y_k|Y_{-\infty:k-1}) \;.
$$
We will provide below some general assumptions, Assumptions~\ref{assum:gen:identif:unique:pi} and~\ref{assum:exist:phi:theta},
that yield precise definitions of $\tilde{\PP}^{\theta}$ and
$p^{\theta,\theta'}(\cdot|\cdot)$.

The second step consists in proving that the set $\Theta_\star$
in~(\ref{eq:ThetaStar-def}) is related to the true parameter $\thv$ in an
exploitable way. Ideally, one could have $\Theta_\star=\{\thv\}$, which would
yield the consistency of $\hat\theta_{\Xinit,n}$ for estimating $\thv$. In this
work, our first objective is to provide a set of general assumptions which ensures that
$\Theta_\star$ is exactly the set of parameters $\theta$ such that
$\tilde\PP^{\theta}=\tilde\PP^{\thv}$. Then this
result guarantees that the estimator converges to the set of parameters
compatible with the true stationary distribution of the observations. If
moreover the model $\nsequence{\tilde\PP^\theta}{\theta\in\Theta}$ is identifiable,
then this set reduces to $\{\thv\}$ and consistency of $\mlY{\Xinit,n}$
directly follows.

To conclude with our general setting, we state the main assumption on the model
and some subsequent notation and definitions used throughout the paper.

\begin{hyp}{K}
\item \label{assum:gen:identif:unique:pi} For all $\theta\in\Theta$,
  the transition kernel $K^\theta$ admits a unique invariant probability
  $\pi^\theta$.
\end{hyp}

We now introduce some important notation used throughout the paper.

\begin{definition}\label{def:equi:theta}
Under Assumption~\ref{assum:gen:identif:unique:pi}, we denote by $\pi_1^\theta$
and $\pi_2^\theta$
the marginal distributions of $\pi^\theta$ on $\Xset$ and $\Yset$, respectively,
and by $\PP^\theta$ and
$\tilde{\mathbb{P}}^\theta$ the probability distributions defined respectively as follows.
\begin{enumerate}[label=\alph*)]
\item $\PP^\theta$ denotes the extension of $\mathbb{P}_{\pi^\theta}^\theta$ on the whole line
$(\Xset\times\Yset)^{\zset}$.
\item $\tilde{\mathbb{P}}^\theta$ is the corresponding projection on the component $\Yset^{\zset}$.
\end{enumerate}
We also use the symbols
$\mathbb{E}^\theta$ and $\tilde{\mathbb{E}}^\theta$ to denote the expectations
corresponding to $\PP^\theta$ and $\tilde{\mathbb{P}}^\theta$, respectively.
Moreover, for all $\theta, \theta'\in\Theta$, we write $\theta\sim\theta'$ if
and only if $\tilde{\mathbb{P}}^{\theta}=\tilde{\mathbb{P}}^{\theta'}$. This defines an
equivalence relation on the parameter set $\Theta$ and the corresponding
equivalence class of $\theta$ is denoted by $[\theta]\eqdef\{\theta' \in\Theta:\; \theta\sim\theta'\}$.
\end{definition}
The equivalence relationship $\sim$ was introduced by \cite{leroux:1992} as an
alternative to the classical identifiability condition.

\subsection{Main result}
\label{subsec:main}
Assumption~\ref{assum:gen:identif:unique:pi} is supposed to hold all along this
section and  $\mathbb{P}^{\theta}$, $\tilde{\mathbb{P}}^{\theta}$ and $\sim$ are
given in \autoref{def:equi:theta}.
Our main result is stated under the following general assumption.

\begin{hyp}{K}
\item \label{assum:exist:phi:theta}
  For all $\theta\neq \theta'$ in $\Theta$, there exists a
  probability kernel $\Phi^{\theta,\theta'}$ on $\Yset^{\zsetn}\times\Xsigma$
  such that for all $A\in\Xsigma$,
%\begin{multline}\label{eq:Phi-def}
$$
\frac{\displaystyle \int_{\Xset}\Phi^{\theta,\theta'}(\chunk{Y}{-\infty}{0};\rmd x_0)\kap{Y_0, Y_1}{x_0; A}}{\displaystyle\int_{\Xset}\Phi^{\theta,\theta'}(\chunk{Y}{-\infty}{0};\rmd x_0)\kap{Y_0, Y_1}{x_0;\Xset}}
=  \Phi^{\theta,\theta'}(\chunk{Y}{-\infty}{1};A), \quad \tilde{\PP}^{\theta'}\as
$$
%\end{multline}
\end{hyp}
\begin{remark}
Note that from \autoref{rem:kappa-is-positive}, the denominator in the \lhs\
of the last displayed equation is strictly positive, which ensures that the ratio is
well defined.
\end{remark}
\begin{remark} \label{rem:intuition:kernel}
Let us give some insight about the formula appearing in \ref{assum:exist:phi:theta} and explain why it is important to consider the cases $\theta=\theta'$ and $\theta\neq \theta'$ separately.
Since $\Xset$ is a Borel space, \cite[Theorem~6.3]{kallenberg2002} applies and the
conditional distribution of $X_0$ given $\chunk{Y}{-\infty}{0}$ under $\PP^\theta$ defines a
probability kernel denoted by $\Phi^{\theta}$. We prove in \autoref{sec:proof-eq10} of \autoref{sec:postponed-proofs} that this kernel satisfies, for all $A \in \Xsigma$,
\begin{equation} \label{eq:phi:theta:theta}
\frac{\displaystyle \int_{\Xset}\Phi^{\theta}(\chunk{Y}{-\infty}{0};\rmd x_0)\kap{Y_0, Y_1}{x_0; A}}{\displaystyle\int_{\Xset}\Phi^{\theta}(\chunk{Y}{-\infty}{0};\rmd x_0)\kap{Y_0, Y_1}{x_0;\Xset}}
=  \Phi^{\theta}(\chunk{Y}{-\infty}{1};A), \quad \tilde{\PP}^{\theta}\as
\end{equation}
Assumption \ref{assum:exist:phi:theta} asserts that the kernel
$\Phi^{\theta,\theta'}$ satisfies a similar identity $\tilde \PP^{\theta'}\as$
for $\theta' \neq \theta$. It is not necessary at this stage to precise how
$\Phi^{\theta,\theta'}$ shows up. This is done in \autoref{subsec:backward}.
\end{remark}
\begin{remark}
The denominator in the ratio displayed in \ref{assum:exist:phi:theta} can be written as
$p^{\theta,\theta'}(Y_1|\chunk{Y}{-\infty}{0})$, where, for all $y\in\Yset$ and $\chunk{y}{-\infty}{0}\in\Yset^{\zsetn}$,
\begin{equation}
\label{eq:def-p_1}
p^{\theta,\theta'}(y|\chunk{y}{-\infty}{0}) \eqdef
\int_{\Xset}\Phi^{\theta,\theta'}(\chunk{y}{-\infty}{0};\rmd x_0)\kap{y_0,
  y}{x_0;\Xset}
\end{equation}
is a conditional density with respect to the measure $\nu$, since for all $(x,y)\in\Xset\times\Yset$, $\int\kap{y,y'}{x;\Xset}\nu(\rmd y')=1$.
\end{remark}

Since $\Yset$ is a Borel space,
\cite[Theorem~6.3]{kallenberg2002} applies and the
conditional distribution of $\chunk{Y}{1}{n}$ given $\chunk{Y}{-\infty}{0}$ defines a
probability kernel. Since $\tilde \PP^\theta(\chunk{Y}{1}{n}\in\cdot)$ is dominated by
$\nu^{\otimes n}$, this
in turns defines a conditional density with respect to $\nu^{\otimes n}$, which we denote by
$p_n^{\theta}(\cdot|\cdot)$, so that for all $B\in\Ysigma^{\otimes n}$,
\begin{equation}\label{eq:p-theta-density}
\tilde \PP^\theta\left(\chunk{Y}{1}{n}\in B\,\vert\,\chunk{Y}{-\infty}{0}\right)=\int_B
p_n^\theta(\chunk{y}{1}{n}|\chunk{Y}{-\infty}{0})\;\nu(\rmd y_1)\cdots\nu(\rmd y_n), \quad\tilde\PP^\theta\as
\end{equation}

% \begin{hyp}{K}
% \item \label{assum:expY1:theta:star} Assumption~\ref{assum:exist:phi:theta}
%   holds and for all $\theta\in\Theta$, under $\tilde\PP^\theta$, $Y_1$ admits
%   the conditional density
%   $p^{\theta}(\cdot|\cdot):=p^{\theta,\theta}(\cdot|\cdot)$ with respect to
%   $\nu$ given $\chunk{Y}{-\infty}{0}$.
% \end{hyp}
Let us now state the main result.
\begin{theorem}\label{thm:mainresult}
Assume that~\ref{assum:gen:identif:unique:pi} holds and define
$\PP^{\theta}$, $\tilde{\PP}^{\theta}$ and $[\theta]$ as in
\autoref{def:equi:theta}. Suppose that Assumption~\ref{assum:exist:phi:theta}
holds.
For all
  $\theta,\theta'\in\Theta$, define $p^{\theta,\theta'}(Y_1|\chunk{Y}{-\infty}{0})$ by~(\ref{eq:def-p_1})
  if $\theta\neq\theta'$ and by $p^{\theta,\theta}(Y_1|\chunk{Y}{-\infty}{0})=p_1^\theta(Y_1|\chunk{Y}{-\infty}{0})$ as
  in~(\ref{eq:p-theta-density}) otherwise.  Then for all $\thv\in\Theta$, we have
\begin{equation}\label{eq:ThetaStar-equiv-class}
\argmax_{\theta\in\Theta}\tilde{\PE}^{\thv}\left[\ln
  p^{\theta,\thv}(Y_1|\chunk{Y}{-\infty}{0})\right] = [\thv]\;.
\end{equation}
% Then for all $\theta\in\Theta_\star$, we have $\theta\sim\thv$.
\end{theorem}
Before proving \autoref{thm:mainresult}, we  first extend the
definition of the conditional density on $\Yset$ in~(\ref{eq:def-p_1}) to a
conditional density on $\Yset^n$.
\begin{definition}\label{def:cond-densities-given-past}
For every positive integer $n$ and $\theta\neq\theta'\in\Theta$, define the function
$p_n^{\theta,\theta'}(\cdot|\cdot)$ on $\Yset^n\times\Yset^{\zsetn}$ by
\begin{equation}
\label{eq:def-p_n}
p_n^{\theta,\theta'}(\chunk{y}{1}{n}|\chunk{y}{-\infty}{0})\eqdef\int_{\Xset^n} \Phi^{\theta,\theta'}(\chunk{y}{-\infty}{0};\rmd x_0)\prod_{k=0}^{n-1}\kap{y_{k}, y_{k+1}}{x_{k}; \rmd x_{k+1}}.
\end{equation}
\end{definition}
Again, it is easy to check that each
$p_n^{\theta,\theta'}(\,\cdot\,|\chunk{y}{-\infty}{0})$ is indeed a density on
$\Yset^n$.  Assumption~\ref{assum:exist:phi:theta} ensures that these density
functions moreover satisfy the successive conditional formula, as for
conditional densities, provided that we restrict ourselves to sequences in a
set of $\tilde{\PP}^{\theta'}$-probability one, as stated in the following lemma.
\begin{lemma}\label{lem:imp:p;seq}
  Suppose that Assumption ~\ref{assum:exist:phi:theta} holds and let
  $p_n^{\theta,\theta'}(\cdot|\cdot)$ be as defined
  in~\autoref{def:cond-densities-given-past}. Then for all $\theta,\theta' \in
  \Theta$ and $n\ge 2$, we have
    \begin{equation}\label{eq:sub:lem:p}
      p_n^{\theta,\theta'} (\chunk{Y}{1}{n}|\chunk{Y}{-\infty}{0})=p_1^{\theta,\theta'} (Y_{n}|\chunk{Y}{-\infty}{n-1})p_{n-1}^{\theta,\theta'} (\chunk{Y}{1}{n-1}|\chunk{Y}{-\infty}{0}), \quad \tilde{\PP}^{\theta'}\as
    \end{equation}
  \end{lemma}
The proof of this lemma is postponed to \autoref{sec:proof-lem:imp}  in
\autoref{sec:postponed-proofs}.
We now have all the tools for proving the main result.
\begin{proof}[Proof of \autoref{thm:mainresult}]
\label{sec:proof-main}
Within this proof section, we will drop the subscript $n$ and respectively write
$p^{\theta,\theta'}(\chunk{y}{1}{n}|\chunk{y}{-\infty}{0})$ and
$p^{\theta}(\chunk{y}{1}{n}|\chunk{y}{-\infty}{0})$  instead of $p^{\theta,\theta'}_n(\chunk{y}{1}{n}|\chunk{y}{-\infty}{0})$  and
$p_n^{\theta}(\chunk{y}{1}{n}|\chunk{y}{-\infty}{0})$ when no
ambiguity occurs.

For all $\theta \in \Theta$, we have by conditioning on $\chunk{Y}{-\infty}{0}$ and by using~(\ref{eq:p-theta-density}),
\begin{align}
&\tilde{\PE}^{\thv}\left[\ln p^{\thv}(Y_1|\chunk{Y}{-\infty}{0})\right]-\tilde{\PE}^{\thv}\left[\ln p^{\theta,\thv}(Y_1|\chunk{Y}{-\infty}{0})\right] \nonumber \\
&\quad=\tilde{\PE}^{\thv}\left[\tilde{\PE}^{\thv}\left[\left. \ln \frac{p^{\thv}(Y_1|\chunk{Y}{-\infty}{0})}{p^{\theta,\thv}(Y_1|\chunk{Y}{-\infty}{0})}\right|\chunk{Y}{-\infty}{0}\right]\right]
\nonumber \\
&\quad=\tilde\PE^{\thv}\left[\mathrm{KL}\left(p_1^{\thv}(\,\cdot\,|\chunk{Y}{-\infty}{0})
\big{\|}p_1^{\theta,\thv}(\,\cdot\,|\chunk{Y}{-\infty}{0})\right)\right]\eqsp, \label{eq:kullback}
\end{align}
where $\mathrm{KL}(p\|q)$ denotes the Kullback-Leibler divergence between the
densities $p$ and $q$. The nonnegativity of the Kullback-Leibler divergence
shows that $\thv$ belongs to the maximizing set on the left-hand side
of~(\ref{eq:ThetaStar-equiv-class}). This implies
\begin{equation}
  \label{eq:ThetaStar-containing}
  \argmax_{\theta\in\Theta}\tilde{\PE}^{\thv}\left[\ln
  p^{\theta,\thv}(Y_1|\chunk{Y}{-\infty}{0})\right] \supseteq [\thv]\;,
\end{equation}
where we have used the following lemma.
\begin{lemma}
  \label{lem:trivial}
Assume that~\ref{assum:gen:identif:unique:pi} holds and define
$\tilde{\PE}^{\theta}$ and $[\theta]$ as in \autoref{def:equi:theta}.
Suppose that for all $\theta\in\Theta$, $G(\theta)$ is a
$\sigma(\chunk{Y}{-\infty}{\infty})$-measurable random variable such that, for
all $\thv\in\Theta$,
$$
\sup_{\theta\in\Theta}\tilde{\PE}^{\thv}\left[G(\theta)\right]=\tilde{\PE}^{\thv}\left[G(\thv)\right] \;.
$$
Then for all $\thv\in\Theta$ and $\theta'\in[\thv]$, we have
$$
\tilde{\PE}^{\thv}\left[G(\theta')\right] = \sup_{\theta\in\Theta}\tilde{\PE}^{\thv}\left[G(\theta)\right] \;.
$$
\end{lemma}
\begin{proof}
  Take $\thv\in\Theta$ and $\theta'\in[\thv]$. Then we have, for all $\theta\in\Theta$,
  $\tilde{\PE}^{\thv}\left[G(\theta)\right]=\tilde{\PE}^{\theta'}\left[G(\theta)\right]$,
  and it follows that
$$
\tilde{\PE}^{\thv}\left[G(\theta')\right]
=\tilde{\PE}^{\theta'}\left[G(\theta')\right] =
\sup_{\theta\in\Theta}\tilde{\PE}^{\theta'}\left[G(\theta)\right] =
\sup_{\theta\in\Theta}\tilde{\PE}^{\thv}\left[G(\theta)\right] \;,
$$
which concludes the proof.
\end{proof}
The proof of the reverse inclusion of~(\ref{eq:ThetaStar-containing}) is more
tricky. Let
us take $\theta\in\Theta_\star$ such that $\theta\neq\thv$ and
show that it implies $\theta\sim\thv$. By \eqref{eq:kullback} we have
$$
\tilde\PE^{\thv}\left[\mathrm{KL}\left(p_1^{\thv}(\,\cdot\,|\chunk{Y}{-\infty}{0})
\big{\|}p_1^{\theta,\thv}(\,\cdot\,|\chunk{Y}{-\infty}{0})\right)\right]=0\eqsp.
$$
 Consequently,
\begin{equation*}
p^{\thv}(Y_1|\chunk{Y}{-\infty}{0}) =
p^{\theta,\thv}(Y_1|\chunk{Y}{-\infty}{0}),\quad \tilde\PP^{\thv}\as
\end{equation*}
Applying ~\autoref{lem:imp:p;seq} and using that $\tilde\PP^{\thv}$ is
shift-invariant, this relation propagates to all $n\geq2$, so that
\begin{equation}\label{eq:eq-ptheta-pthetastar}
p^{\thv}\left(\chunk{Y}{1}{n}\vert \chunk{Y}{-\infty}{0}\right)=p^{\theta,\thv}\left(\chunk{Y}{1}{n}\vert \chunk{Y}{-\infty}{0}\right), \quad \tilde{\PP}^{\thv}\as
\end{equation}
For any measurable function $H:\Yset^n\to\rset_+$, we get
\begin{align*}
\tilde\PE^{\thv}\left[H(\chunk{Y}{1}{n})\right]%=\tilde\PE^{\thv}\left[\tilde\PE^{\thv}\left(H(\chunk{Y}{1}{n})\vert \chunk{Y}{-\infty}{0}\right)\right]\eqsp, \\
&=\tilde\PE^{\thv}\left\{\tilde\PE^{\thv}\left[H(\chunk{Y}{1}{n})\frac{p^{\theta,\thv}(\chunk{Y}{1}{n}\vert \chunk{Y}{-\infty}{0})}{p^{\thv}(\chunk{Y}{1}{n}\vert \chunk{Y}{-\infty}{0})}\middle | \chunk{Y}{-\infty}{0}\right]\right\} \\
&=\tilde\PE^{\thv}\left[\int H(\chunk{y}{1}{n})p^{\theta,\thv}(\chunk{y}{1}{n}|\chunk{Y}{-\infty}{0})\nu^{\otimes n}(\rmd \chunk{y}{1}{n})\right]\eqsp,
\end{align*}
where the last equality follows from~(\ref{eq:p-theta-density}).
Using~\autoref{def:cond-densities-given-past} and Tonelli's theorem, we obtain
\begin{align*}
\tilde\PE^{\thv}\left[H(\chunk{Y}{1}{n})\right]
&=\tilde\PE^{\thv}\int H(\chunk{y}{1}{n})\int \Phi^{\theta,\thv}(\chunk{Y}{-\infty}{0};\rmd x_0)\kap{Y_0, y_1}{x_0; \rmd x_{1}}\times\\
&\quad\quad\prod_{k=1}^{n-1}\kap{y_{k}, y_{k+1}}{x_{k}; \rmd x_{k+1}}\nu^{\otimes n}(\rmd \chunk{y}{1}{n})\eqsp, \\
&=\tilde\PE^{\thv}\int \Phi^{\theta,\thv}(\chunk{Y}{-\infty}{0};\rmd x_0)\int H(\chunk{y}{1}{n})\kap{Y_0, y_1}{x_0; \rmd x_{1}}\times\\
&\quad\quad \prod_{k=1}^{n-1}\kap{y_{k}, y_{k+1}}{x_{k}; \rmd x_{k+1}})\nu^{\otimes n}(\rmd \chunk{y}{1}{n})\eqsp, \\
&= \tilde\PE^{\thv}\int\Phi^{\theta,\thv}(\chunk{Y}{-\infty}{0};\rmd x_0)\PE^{\theta}_{(x_0, Y_0)}\left[H(\chunk{Y}{1}{n})\right]\eqsp,\\
&=\PE^{\theta}_{\pi^{\theta, \thv}}\left[H(\chunk{Y}{1}{n})\right]\eqsp,
\end{align*}
where $\pi^{\theta, \thv}$ is a probability on $\Xset\times\Yset$
defined by
\begin{equation*}
\pi^{\theta, \thv}(A\times B)\eqdef\tilde\PE^{\thv}\left[\Phi^{\theta,\thv}(\chunk{Y}{-\infty}{0};A)\1_B(Y_0)\right]\eqsp,
\end{equation*}
for all $(A, B)\in\Xsigma\times\Ysigma$. Consequently, for all
$B\in\Ysigma^{\otimes\zsetpnz}$,
\begin{equation}\label{eq:eq-ptheta-pthetastar1}
\tilde\PP^{\thv}(\Yset^{\zsetn}\times B)=\PP^{\theta}_{\pi^{\theta,
    \thv}}(\Xset^{\zsetp}\times (\Yset\times B))\eqsp.
\end{equation}
If we had $\pi^{\theta}=\pi^{\theta, \thv}$, then we could conclude that the
two shift-invariant distributions $\tilde\PP^{\thv}$ and $\tilde\PP^{\theta}$
are the same and thus $\theta\sim\thv$. Therefore, to complete the proof, it only
remains to show that $\pi^{\theta}=\pi^{\theta, \thv}$, which
by~\ref{assum:gen:identif:unique:pi} is equivalent to showing that $\pi^{\theta, \thv}$
is an invariant distribution for $K^\theta$.

Let us now prove this latter fact. Using that $\tilde\PP^{\thv}$ is shift-invariant and then
conditioning on $\chunk{Y}{-\infty}{0}$, we have, for any $(A, B)\in\Xsigma\times\Ysigma$,
\begin{align*}
\pi^{\theta, \thv}(A\times B)%&=\tilde{\PE}^{\thv}\left[\Phi^{\theta,\thv}(\chunk{Y}{-\infty}{0};A)\1_B(Y_0)\right]\eqsp, \\
&=\tilde{\PE}^{\thv}\left[\Phi^{\theta,\thv}(\chunk{Y}{-\infty}{1};A)\1_B(Y_1)\right]\eqsp,\\
&=\tilde{\PE}^{\thv}\int \Phi^{\theta,\thv}(\chunk{Y}{-\infty}{0},y_1;A)\1_B(y_1)\,p^{\thv}(y_1|\chunk{Y}{-\infty}{0})\,\nu(\rmd y_1)\eqsp,\\
&=\tilde{\PE}^{\thv}\int \Phi^{\theta,\thv}(\chunk{Y}{-\infty}{0},y_1;A)\1_B(y_1)\,p^{\theta,\thv}(y_1|\chunk{Y}{-\infty}{0})\,\nu(\rmd y_1)\eqsp,
\end{align*}
where in the last equality we have used ~\eqref{eq:eq-ptheta-pthetastar}.
Using~\ref{assum:exist:phi:theta} we then get
\begin{align*}
&\pi^{\theta, \thv}(A\times B)\\
&=\tilde{\PE}^{\thv}\int\Phi^{\theta,\thv}(\chunk{Y}{-\infty}{0};\rmd x_0)\kap{Y_0, y_1}{x_0; \rmd x_{1}}\1_A(x_1)\1_B(y_1)\nu(\rmd y_1)\eqsp,\\
&=\tilde{\PE}^{\thv}\int\Phi^{\theta,\thv}(\chunk{Y}{-\infty}{0};\rmd x_0)
K^\theta((x_0, Y_0);A\times B)\eqsp,\\
&=\pi^{\theta, \thv}K^\theta(A\times B)\eqsp.
\end{align*}
Thus, $\pi^{\theta, \thv}$ is an invariant distribution for $K^\theta$, which
concludes the proof.

\end{proof}

\subsection{Construction of the kernel $\Phi^{\theta,\theta'}$ as a backward limit}
\label{subsec:backward}
Again, all along this section, Assumption~\ref{assum:gen:identif:unique:pi} is
supposed to hold and the symbols $\mathbb{P}^{\theta}$ and
$\tilde{\mathbb{P}}^{\theta}$ refer to the probabilities introduced in
\autoref{def:equi:theta}.  In addition to
Assumption~\ref{assum:gen:identif:unique:pi}, \autoref{thm:mainresult}
fundamentally relies on Assumption~\ref{assum:exist:phi:theta}. These
assumptions ensure the existence of the probability kernel $\Phi^{\theta,
  \theta'}$ that yields the definition of $p_1^{\theta,\theta'}(\cdot|\cdot)$.
We now explain how the kernel $\Phi^{\theta,\theta'}$ may arise as a limit
under $\PP^{\theta'}$ of explicit kernels derived from $K^\theta$.  It will
generally apply to observation-driven models, treated in
\autoref{sec:observ-driv-model}, but also in the more classical case of HMMs, as explained in \autoref{sec:hmm}. A natural approach
is to define the kernel $\Phi^{\theta, \theta'}$ as the weak limit of the
following ones.
\begin{definition} \label{def:Phi:x:f} Let $n$ be a positive integer. For all
  $\theta\in\Theta$ and $x\in\Xset$, we define the probability kernel
  $\Phi_{x,n}^\theta$ on $\Yset^{n+1}\times\Xsigma$ by, for all
  $\chunk{y}{0}{n}\in\Yset^{n+1}$ and $A\in\Xsigma$,
$$
\Phi_{x,n}^{\theta}(\chunk{y}{0}{n};A) \eqdef
\frac{\displaystyle \int_{\Xset^{n-1}\times A}\; \prod_{k=0}^{n-1} \kap{y_{k},
    y_{k+1}}{x_{k}; \rmd x_{k+1}}}{\displaystyle \int_{\Xset^n}
  \;\prod_{k=0}^{n-1}\kap{y_{k}, y_{k+1}}{x_{k}; \rmd x_{k+1}}}
\quad\text{with $x_0=x$.}
$$
We will drop the subscript $n$ when no ambiguity occurs.
\end{definition}
It is worth noting that $\Phi_{x,n}^{\theta}(\chunk{Y}{0}{n};\cdot)$ is the conditional distribution of
$X_n$ given $\chunk{Y}{1}{n}$ under $\PP^\theta_{(x,Y_0)}$. To derive the desired $\Phi^{\theta, \theta'}$ we take, for a well-chosen $x$, the
limit of $\Phi_{x,n}^{\theta}(\chunk{y}{0}{n};\cdot)$ as $n\to\infty$ for a
sequence $\chunk{y}{0}{n}$ corresponding to a path under
$\tilde\PP^{\theta'}$. The precise statement is provided in Assumption~\ref{assum:help:exist:phi} below,  which requires the following definition. For
all $\theta\in\Theta$ and for all nonnegative measurable functions $f$ defined on
$\Xset$, we set
$$
\mathcal{F}_f^\theta \eqdef \left\{x \mapsto \kap{y, y'}{x; f}:\; (y,y')\in\Yset^2\right\} \;.
$$
We can now state the assumption as follows.
\begin{hyp}{K}
\item \label{assum:help:exist:phi} For all $\theta\neq\theta'\in\Theta$, there
  exist $x\in\Xset$, a probability kernel $\Phi^{\theta, \theta'}$ on
  $\Yset^{\zsetn}\times\Xsigma$ and a countable class $\mathcal{F}$ of
  $\Xset\to\rsetp$ measurable functions such that for all
  $f\in\mathcal{F}$,
\begin{align*}
\tilde{\PP}^{\theta'}\left(\forall f'\in\mathcal{F}_f^\theta\cup \{f\},\; \lim_{m\to\infty}\Phi_{x,m}^\theta (\chunk{Y}{-m}{0};f') = \Phi^{\theta,\theta'}
(\chunk{Y}{-\infty}{0};f') <\infty\right) = 1\;.
\end{align*}
\end{hyp}

The next lemma shows that, provided that $\mathcal{F}$ is rich enough,
Assumption~\ref{assum:help:exist:phi}  can be directly
used to obtain Assumption~\ref{assum:exist:phi:theta}.
In what follows, we say that a class of $\Xset\to\rset$ functions is separating
if,  for any two probability measures $\mu_1$ and $\mu_2$ on $(\Xset, \Xsigma)$,
the equality of $\mu_1(f)$ and $\mu_2(f)$ over $f$ in the class implies the equality of
the two measures.

\begin{lemma}\label{lem:help:I2}
  Suppose that Assumption~\ref{assum:help:exist:phi} holds and that
  $\mathcal{F}$ is a separating class of functions containing $\1_\Xset$. Then the kernel $\Phi^{\theta,\theta'}$ satisfies
  Assumption~\ref{assum:exist:phi:theta}.
\end{lemma}

\begin{proof}
 Let $x\in\Xset$ be given in
 Assumption~\ref{assum:help:exist:phi}. From~\autoref{def:Phi:x:f}, we may
 write, for all $f\in\mathcal{F}$, setting  $x_{-m}=x$,
\begin{equation*}
 \Phi_{x,m}^{\theta}(\chunk{Y}{-m}{0};f)=\frac{\displaystyle \int f(x_0)\prod_{k=-m}^{-1} \kap{Y_{k}, Y_{k+1}}{x_{k}; \rmd x_{k+1}}}{\displaystyle \int \prod_{k=-m}^{-1} \kap{Y_{k}, Y_{k+1}}{x_{k}; \rmd x_{k+1}}}
\end{equation*}
and, similarly,
\begin{equation}
 \Phi_{x,m+1}^{\theta}(\chunk{Y}{-m}{1};f)=\frac{\displaystyle \int f(x_1)\prod_{k=-m}^{0} \kap{Y_{k}, Y_{k+1}}{x_{k}; \rmd x_{k+1}}}{\displaystyle \int \prod_{k=-m}^{0} \kap{Y_{k}, Y_{k+1}}{x_{k}; \rmd x_{k+1}}}\eqsp.\label{eq:consq:Phi}
\end{equation}
Dividing both numerator and denominator of~\eqref{eq:consq:Phi} by $$\int \prod_{k=-m}^{-1} \kap{Y_{k}, Y_{k+1}}{x_{k}; \rmd x_{k+1}}\;,$$ which is strictly positive by~\autoref{rem:kappa-is-positive}, then~\eqref{eq:consq:Phi} can be rewritten as
\begin{equation} \label{eq:recurr:m}
\Phi_{x,m+1}^{\theta}(\chunk{Y}{-m}{1};f)=\frac{\Phi_{x,m}^\theta\left(\chunk{Y}{-m}{0};\kap{Y_0, Y_1}{\cdot; f}\right)}{\Phi_{x,m}^\theta\left(\chunk{Y}{-m}{0};\kap{Y_0, Y_1}{\cdot; \1_\Xset}\right)}\eqsp.
\end{equation}
Letting $m\to\infty$ and applying Assumption~\ref{assum:help:exist:phi}, then $\tilde\PP^{\theta'}\as$,
\begin{align*}
\Phi^{\theta, \theta'}(\chunk{Y}{-\infty}{1};f)&=\frac{\Phi^{\theta, \theta'}\left(\chunk{Y}{-\infty}{0};\kap{Y_0, Y_1}{\cdot; f}\right)}{\Phi^{\theta, \theta'}\left(\chunk{Y}{-\infty}{0};\kap{Y_0, Y_1}{\cdot; \1_\Xset}\right)}\eqsp,\\
&=\frac{\displaystyle\int \Phi^{\theta, \theta'}\left(\chunk{Y}{-\infty}{0};\rmd x_0\right)\kap{Y_0, Y_1}{x_0; f}}{\displaystyle\int \Phi^{\theta, \theta'}\left(\chunk{Y}{-\infty}{0};\rmd x_0\right)\kap{Y_0, Y_1}{x_0; \1_\Xset}}\eqsp.
\end{align*}
Since $\mathcal{F}$ is a separating class, the proof is concluded.

\end{proof}

\section{Application to hidden Markov models}
\label{sec:hmm}

\subsection{Definitions and assumptions}

 Hidden Markov models belong to a subclass of partially observed Markov models
defined as follows.
\begin{definition}
  Consider a partially observed and partially dominated Markov model given
  in~\autoref{def:partially-dom} with Markov kernels
  $\nsequence{K^\theta}{\theta\in\Theta}$. We will say that this model is a  hidden Markov
  model if the kernel $K^\theta$ satisfies
\begin{equation}\label{eq:defK:hmm}
K^\theta((x, y); \rmd x'\rmd y')=Q^\theta(x;\rmd x')G^\theta(x';\rmd y') \eqsp.
\end{equation}
Moreover, in this context, we always assume that $(\Xset,\Xmet)$ is a complete
separable metric space and $\Xsigma$ denotes the associated Borel
$\sigma$-field.
\end{definition}
In \eqref{eq:defK:hmm}, $Q^\theta$ and $G^\theta$ are transition kernels on
$\Xset \times \Xsigma$ and $\Xset \times \Ysigma$, respectively. Since the model
is partially dominated, we denote by $g^\theta$ the corresponding Radon-Nikodym
derivative of $G^\theta(x;\cdot)$ with respect to the dominating measure $\nu$:
for all $(x,y) \in \Xset \times \Yset$,
$$
\frac{\rmd G^\theta(x;\cdot)}{\rmd \nu} (y)=g^\theta(x;y)\eqsp.
$$
One can directly observe that the unnormalized kernel $\kap{y,y'}{}$ defined in \eqref{eq:kappa} does no longer depend on $y$, and in this case, one can write
\begin{equation}
  \label{eq:kappa-hmm}
\kap{y,y'}{x;\rmd x'}=\kap{y'}{x;\rmd x'} = Q^\theta(x;\rmd x')g^\theta(x';y') \eqsp.
\end{equation}
For any integer $n\geq 1$, $\theta \in \Theta$ and sequence
$\chunk{y}{0}{n-1} \in \Yset^n$, consider the unnormalized kernel
$\Pblock[\theta]{\chunk{y}{0}{n-1}}$ on $\Xset\times\Xsigma$ defined by, for all $x_0
\in \Xset$ and $A \in \Xsigma$,
\begin{equation}
\label{eq:def-Pblock}
\Pblock[\theta]{\chunk{y}{0}{n-1}}(x_0;A)= \idotsint  \left[\prod_{k=0}^{n-1} g^\theta(x_{k};y_{k}) Q^\theta(x_{k}; \rmd x_{k+1}) \right] \1_A (x_{n}) \eqsp,
\end{equation}
so that the MLE $\mlY{\Xinit,n}$, associated to the observations
$\chunk{Y}{0}{n-1}$ with an arbitrary initial distribution $\Xinit$ on $\Xset$
is defined by
$$
\mlY{\Xinit,n} \in \argmax_{\theta \in \Theta} \Xinit \Pblock[\theta]{\chunk{Y}{0}{n-1}} \1_\Xset\eqsp.
$$
We now follow the approach taken by \cite{douc:moulines:2012} in misspecified
models and show that in the context of well-specified models, the maximizing
set of the asymptotic normalized log-likelihood can be identified by relying
neither on the exponential separation of measures, nor on the rates of
convergence to the equilibrium, but only on the uniqueness of the invariant
probability. We note the following fact which can be used to check
\ref{assum:gen:identif:unique:pi}.
\begin{remark}\label{rem:A1-hmm}
In the HMM context, $\pi^\theta$ is an
invariant distribution of $K^\theta$ if and only if $\pi_1^\theta$ is  an
invariant distribution of $Q^\theta$ and $\pi^\theta(\rmd x\rmd y) =\pi_1^\theta(\rmd x)G^\theta(x;\rmd y)$.
\end{remark}
We illustrate the application of the main result (\autoref{thm:mainresult}) in
the context of HMMs by considering the assumptions of
\cite{douc:moulines:2012} in the particular case of blocks of size 1
($r=1$). Of course, general assumptions with arbitrary sizes of blocks could
also be used but this complicates significantly the expressions and may confine
the attention of the reader to unnecessary technicalities. To keep the
discussion simple, we only consider blocks of size 1, which already covers many
cases of interest.

Before listing the main assumptions, we recall the definition of a so-called \emph{local
Doeblin set} (in the particular case where $r=1$) as introduced in
\cite[Definition 1]{douc:moulines:2012}.
\begin{definition} \label{defi:local-Doeblin:one} A set $C$ is local Doeblin
  \wrt\ the family of kernels $\nsequence{Q^\theta}{\theta \in \Theta}$ if
  there exist positive constants $\epsilon^-_C,\ \epsilon^+_C$ and a family of
  probability measures $\nsequence{\lambda_{C}^\theta}{\theta \in \Theta}$ such that,
  for any $\theta \in \Theta$ , $\lambda_{C}^\theta(C) =1$, and, for any $A \in
  \Xsigma$ and $x \in C$,
\begin{equation*}
\epsilon^{-}_C \lambda_{C}^\theta (A) \leq Q^\theta(x; A \cap C) \leq \epsilon^{+}_{C} \lambda_{C}^\theta(A)
 \eqsp.
\end{equation*}
\end{definition}
Consider now the following set of assumptions.
\begin{hyp}{D}
\item \label{assum:q:positive} There exists a  $\sigma$-finite  measure $\mu$ on
  $(\Xset,\Xsigma)$ that dominates $Q^\theta(x; \cdot)$ for all $(x,\theta) \in
  \Xset \times \Theta$. Moreover, denoting $q^\theta(x; x')\eqdef\frac{\rmd
    Q^\theta(x; \cdot)}{\rmd \mu}(x')$, we have
$$
q^\theta(x; x')>0\,, \quad \mbox{for all }(x,x',\theta) \in \Xset \times \Xset\times \Theta\eqsp.
$$
\item\label{assum:g-unif-bounded} For all $y\in\Yset$, we have $\displaystyle\sup_{\theta\in\Theta}\sup_{x\in\Xset}g^\theta(x;y)<\infty$.
\item \label{assum:fullyDominated:filter}
\begin{enumerate} [label=(\alph*)]
\item \label{item:condition-L-K} %\item \label{item:bound-eta-G}
For all $\thv\in\Theta$, there exists a set $K \in \Ysigma$ with
$\tilde \PP^\thv(Y_0 \in K) > 2 / 3$ such that for all $\eta > 0$, there exists a local Doeblin
  set $C \in \Xsigma$ \wrt\ $\nsequence{Q^\theta}{\theta \in \Theta}$
  satisfying,  for all $\theta \in \Theta$ and all $y \in K$,
\begin{equation}
\label{eq:bound-eta-G}
\sup_{x \in C^c} g^\theta(x; y) \leq \eta \sup_{x \in \Xset} g^\theta(x; y) < \infty \eqsp.
\end{equation}
\item \label{item:mino-g-simple} For
  all $\thv\in\Theta$, there exists a set $D \in \Xsigma$ satisfying
$$
\inf_{\theta \in \Theta}\inf_{x \in D} Q^\theta(x; D) > 0 \quad \mbox{and}\quad \tilde\PE^\thv\left[\ln^- \inf_{\theta \in \Theta} \inf_{x \in D} g^\theta(x; Y_0)\right] < \infty \eqsp.
$$
\end{enumerate}
\item \label{assum:majo-g}
For
  all $\thv\in\Theta$,  $\displaystyle\tilde \PE^\thv\left[ \lnp\sup_{\theta \in \Theta} \sup_{x \in \Xset} g^\theta(x;Y_0)\right]<\infty$.
\item \label{assum:continuity}
There exists $p \in \zsetp$ such that for any $x \in \Xset$ and $n \geq p$,
the function $\theta \mapsto \Pblock[\theta]{\chunk{Y}{0}{n}}(x;\Xset)$
is $\tilde \PP^{\thv}\as$ continuous on $\Theta$.
\end{hyp}

\begin{remark}\label{rem:irred-hmm-D1}
Under~\ref{assum:q:positive}, for
all $\theta\in\Theta$, the Markov kernel $Q^\theta$ is $\mu$-irreducible, so that,
using \autoref{rem:A1-hmm},~\ref{assum:gen:identif:unique:pi} reduces to the
existence of a stationary distribution for $Q^\theta$.
\end{remark}
\begin{remark}
  Assumptions \ref{assum:fullyDominated:filter},~\ref{assum:majo-g} and~\ref{assum:continuity}
  and~(\ref{eq:g-evrywhere-strictly-positive}) in \autoref{def:partially-dom}
  correspond to (A1), (A2) and (A3) in \cite{douc:moulines:2012},
  where the blocks are of size $r=1$.
\end{remark}
\begin{remark}
  Assumption~\ref{assum:majo-g} implies~\ref{assum:g-unif-bounded} up to a
  modification of $g^\theta(x;y)$ on $\nu$-negligible set of $y \in
  \Yset$ for all $x\in\Xset$.  Indeed, \ref{assum:majo-g} implies that
  $\sup_\theta \sup_x g^\theta(x;Y_0)<\infty$, $\tilde \PP^\thv\as$, and it can
  be shown that under~\ref{assum:q:positive}, $\pi_2^\thv=\pi^\thv(\Xset\times\cdot)$ is equivalent to
  $\nu$ for all $\theta\in\Theta$.
\end{remark}
In these models, the kernel $\Phi_{x,n}^\theta $ introduced in \autoref{def:Phi:x:f} writes
$$
\Phi_{x,n}^{\theta}(\chunk{y}{1}{n};A) =
\frac{\displaystyle \int_{\Xset^{n-1}\times A}\; \prod_{k=0}^{n-1} Q^{\theta}(x_{k}; \rmd x_{k+1})g^{\theta}(x_{k+1};y_{k+1})}{\displaystyle \int_{\Xset^n}
  \;\prod_{k=0}^{n-1} Q^{\theta}(x_{k}; \rmd x_{k+1})g^{\theta}(x_{k+1};y_{k+1})} \quad\text{with $x_0=x$}\eqsp.
$$
The distribution $\Phi_{x,n}^\theta (\chunk{Y}{0}{n};\cdot)$ is usually
referred to as the {\em filter distribution}.
\autoref{prop:limiteRelativeEntropyRateY} (below) can be derived from
\cite[Proposition~1]{douc:moulines:2012}. For blocks of size $1$, the initial
distributions in  \cite{douc:moulines:2012} are constrained to belong to the
set $\mathcal{M}^\thv(D)$ of all probability distributions $\xi$ defined on
$(\Xset,\Xsigma)$ such that
\begin{equation}
  \label{eq:M-D-cond}
\tilde \PE^\thv\left[\ln^- \inf_{\theta \in \Theta} \int \Xinit(\rmd x) g^\theta(x;Y_0) Q^\theta(x;D) \right] < \infty\eqsp,
\end{equation}
where $D \in \Xsigma$ is the set appearing in
\ref{assum:fullyDominated:filter}.  It turns out that under
\ref{assum:fullyDominated:filter}-\ref{item:mino-g-simple}, all probability
distributions $\Xinit$ satisfy~(\ref{eq:M-D-cond}), so the constraint on the
initial distribution vanishes in our case.
\begin{proposition} \label{prop:limiteRelativeEntropyRateY}
Assume \ref{assum:fullyDominated:filter} and~\ref{assum:majo-g}. Then the
following assertions hold.
\begin{enumerate}[label=(\roman*)]
\item \label{item:lim-ponct-cond-rate-entropy-Y} For any $\theta, \thv \in
  \Theta$, there exists a probability kernel $\Phi^{\theta,\thv}$ on
  $\Yset^{\zsetn}\times\Xsigma$ such that for any $x \in \Xset$,
  $$
     \tilde \PP^{\thv}\left( \mbox{for all bounded $f$},\; \lim_{m \to \infty}\Phi_{x,m}^\theta(\chunk{Y}{-m}{0};f)=\Phi^{\theta,\thv}(\chunk{Y}{-\infty}{0};f)\right)=1\eqsp.
     $$
\item \label{item:lim-ponct-rate-entropy-Y} For any $\theta,\thv \in \Theta$ and
  probability measure $\Xinit$,
$$
\lim_{n \to\infty} n^{-1} \ln \Xinit \Pblock[\theta]{\chunk{Y}{0}{n-1}}\1_\Xset= \ell(\theta,\thv),\quad \PP^{\thv}\as\eqsp,
$$
where
\begin{equation}
\label{eq:def:ell:hmm}
\ell(\theta,\thv) \eqdef \tilde\PE^{\thv}\left[\ln \int \Phi^{\theta,\thv}(\chunk{Y}{-\infty}{0};\rmd x_0) \kap{Y_1}{x_0;\Xset}\right]\;.
\end{equation}
\end{enumerate}
\end{proposition}

\subsection{Equivalence-class consistency}

We can now state the main result on the consistency of the MLE for HMMs.
\begin{theorem} \label{thm:consistanceY} Assume
  that~\ref{assum:gen:identif:unique:pi} holds and define
  $\mathbb{P}^{\theta}$, $\tilde{\mathbb{P}}^{\theta}$ and the equivalence
  class $[\theta]$ as in \autoref{def:equi:theta}.  Moreover, suppose that
  $(\Theta,\met)$ is a compact metric space and that
  Assumptions \ref{assum:q:positive}--\ref{assum:continuity} hold. Then, for any
  probability measure $\Xinit$,
\begin{equation*}
\lim_{n \to \infty} \met(\mlY{\Xinit,n},[\thv])=0,\quad \tilde \PP^\thv\as
\end{equation*}
\end{theorem}

\begin{proof}
According to \cite[Theorem 2]{douc:moulines:2012},
$\theta\mapsto\ell(\theta,\thv)$ defined
by~(\ref{eq:def:ell:hmm}) is upper semi-continuous
(so that $\Theta_\star:=\argmax_{\theta \in \Theta}\ell(\theta,\thv)$ is
non-empty) and moreover
$$
\lim_{n \to \infty} \met(\mlY{\Xinit,n},\Theta_\star)=0\,, \quad \tilde \PP^\thv\as
$$
The proof then follows from \autoref{thm:mainresult}, provided that
$\ell(\theta,\thv)$ can be expressed as in the statement of
\autoref{thm:mainresult} and that \ref{assum:exist:phi:theta} is
satisfied. First note that, for $\theta\neq\thv$, the integral appearing within the logarithm in
\eqref{eq:def:ell:hmm} corresponds to
$p^{\theta,\thv}(Y_1|\chunk{Y}{-\infty}{0})$ with $p^{\theta,\thv}$ as defined
in~(\ref{eq:def-p_1}).

Let $\mcf$ be a countable separating class of nonnegative bounded functions
containing $\1_\Xset$, see~\cite[Theorem~6.6,~Chapter~6]{parthasarathy:2005}
for the existence of such a class. By \autoref{lem:help:I2}, we check
\ref{assum:exist:phi:theta} by showing that \ref{assum:help:exist:phi} is
satisfied. Condition~\ref{assum:g-unif-bounded} and~(\ref{eq:kappa-hmm}) imply
that for all bounded functions $f$, $\mathcal{F}_f^\theta$ is a class of bounded
functions, and this in turn implies \ref{assum:help:exist:phi} by applying
\autoref{prop:limiteRelativeEntropyRateY}-\ref{item:lim-ponct-cond-rate-entropy-Y}
to all $x$. Thus, \ref{assum:exist:phi:theta}
is satisfied, and for $\theta\neq \thv$, $\ell(\theta,\thv)$ can be expressed
as in the statement of \autoref{thm:mainresult}. To complete the proof, it only
remains to consider the case where $\theta=\thv$ and to show that $\ell(\thv,\thv)$ can
be written as
\begin{equation}
  \label{eq:limit-likelihood-hmm-trueparam}
\ell(\thv,\thv) = \tilde\PE^{\thv}\left[\ln p_1^\thv(Y_1|\chunk{Y}{-\infty}{0})\right] \eqsp,
\end{equation}
where $p_1^\thv(\cdot|\cdot)$ is the conditional density given in~\eqref{eq:p-theta-density}.
According to \cite[Theorem 1]{barron:1985}, we have
\begin{equation}\label{eq:hmm:first}
\tilde\PE^{\thv}\left[\ln p_1^\thv(Y_1|\chunk{Y}{-\infty}{0})\right]= \lim_{n \to \infty} n^{-1} \ln \pi_1^\thv \Pblock[\thv]{\chunk{Y}{0}{n-1}}\1_\Xset\,, \quad \tilde\PP^\thv\as
\end{equation}
On the other hand, applying \autoref{prop:limiteRelativeEntropyRateY}-\ref{item:lim-ponct-rate-entropy-Y} yields
\begin{equation}\label{eq:hmm:second}
\ell(\thv, \thv)=  \lim_{n \to\infty} n^{-1} \ln \Xinit \Pblock[\thv]{\chunk{Y}{0}{n-1}}\1_\Xset,\quad \tilde\PP^{\thv}\as
\end{equation}
Observe that, by using~\ref{assum:q:positive}, the probability measure
$\Xinit\Pblock[\thv]{y_0}$ admits a density with respect to $\mu$ given by
  \begin{equation}
    \label{eq:densityx1y0}
 \frac{\rmd \Xinit\Pblock[\thv]{y_0}}{\rmd \mu}(x_1)=\int \Xinit(\rmd x_0)g^{\thv}(x_0;y_0)\,q^{\thv}(x_0; x_1)\;.
  \end{equation}
We further get, for all $\chunk{y}{0}{n-1}\in\Yset^n$,
$$
\Xinit \Pblock[\thv]{\chunk{y}{0}{n-1}}\1_\Xset = \int  \frac{\rmd
  \Xinit\Pblock[\thv]{y_0}}{\rmd \mu}(x_1) \times
\left(\delta_{x_1}\Pblock[\thv]{\chunk{y}{1}{n-1}}\1_{\Xset}\right) \;\mu(\rmd x_1)\;,
$$
and under $\PP^\thv$, the joint density of $(X_1,\chunk{Y}{0}{n-1})$ with respect
to $\mu\otimes\nu^{\otimes n}$ is given by
$$
p^{\thv}_{1,n}(x_1,\chunk{y}{0}{n-1}):=  \frac{\rmd \pi_1^\thv\Pblock[\thv]{y_0}}{\rmd
  \mu}(x_1)%\pi_1^\thv(\rmd x_0)g^{\thv}(x_0,y_0)\,q^{\thv}(x_0,x_1)
\times
\left(\delta_{x_1}\Pblock[\thv]{\chunk{y}{1}{n-1}}\1_{\Xset}\right) \;.
$$
Note that we similarly have, for all $y_0\in\Yset$ and $x_1\in\Xset$,
  \begin{equation}
    \label{eq:densityx1y0-stat}
 \frac{\rmd \pi_1^\thv\Pblock[\thv]{y_0}}{\rmd \mu}(x_1)=\int \pi_1^\thv(\rmd x_0)g^{\thv}(x_0; y_0)\,q^{\thv}(x_0; x_1)\;.
\end{equation}
The four previous displays yield, for all $\chunk{y}{0}{n-1}\in\Yset^n$,
\begin{multline*}
\Xinit \Pblock[\thv]{\chunk{y}{0}{n-1}}\1_\Xset \\=
\int\frac{\int \Xinit(\rmd x_0) g^\thv(x_0; y_0) q^\thv(x_0; x_1)}{\int
  \pi_1^\thv(\rmd x_0) g^\thv(x_0; y_0) q^\thv(x_0; x_1)}
p^{\thv}_{1,n}(x_1,\chunk{y}{0}{n-1}) \;\mu(\rmd x_1)\;.
\end{multline*}
Dividing by the density of $\chunk{Y}{0}{n-1}$  with respect
to $\nu^{\otimes n}$   under $\PP^\thv$, we get
$$
\frac{\Xinit \Pblock[\thv]{\chunk{Y}{0}{n-1}}\1_\Xset}{\pi_1^\thv \Pblock[\thv]{\chunk{Y}{0}{n-1}}\1_\Xset}=
\CPEu[\thv]{}{R(X_1,Y_0)}{\chunk{Y}{0}{n-1}}\eqsp,\quad\tilde\PP^\thv\as\;,
$$
where $R(x_1,y_0)$ is the ratio between~(\ref{eq:densityx1y0})
and~(\ref{eq:densityx1y0-stat}), which are positive densities with respect
to $\mu\otimes\nu$. Since the denominator~(\ref{eq:densityx1y0-stat}) is the density of $(X_1,Y_0)$ under $\PP^\thv$, we then have
$$
\PE^{\thv}[R(X_1,Y_0)]= 1\eqsp.
$$
By L\'{e}vy's zero-one law, we thus get that
$$
\lim_{n\to\infty}\frac{\Xinit \Pblock[\thv]{\chunk{Y}{0}{n-1}}\1_\Xset}{\pi_1^\thv \Pblock[\thv]{\chunk{Y}{0}{n-1}}\1_\Xset}=\CPEu[\thv]{}{R(X_1,Y_0)}{\chunk{Y}{0}{\infty}}\eqsp,\quad\tilde\PP^{\thv}\as\eqsp,
$$
and since by \ref{assum:q:positive}, $R(x_1,y_0)$ takes only positive values, this limit is thus positive. This implies
$$
\lim_{n\to\infty} n^{-1} \ln \frac{\Xinit \Pblock[\thv]{\chunk{Y}{0}{n-1}}\1_\Xset}{\pi_1^\thv \Pblock[\thv]{\chunk{Y}{0}{n-1}}\1_\Xset}=0 \quad \tilde\PP^\thv\as
$$
Combining with \eqref{eq:hmm:first} and \eqref{eq:hmm:second}, we finally
obtain~(\ref{eq:limit-likelihood-hmm-trueparam}), which concludes the proof.
\end{proof}

\subsection{A polynomially ergodic example}
As an application of \autoref{thm:consistanceY}, we consider
the HMM model described in \autoref{exple:hmm}. In addition to the assumptions
introduced in  \autoref{exple:hmm}, we assume that $U_0$ and $V_0$ are independent
and centered and they both admit densities with respect to the Lebesgue measure
$\lambda$ over $\rset$, denoted by $r$ and $h$, respectively, and
\begin{hyp}{E}
\item \label{item:q} the density $r$ satisfies:
\begin{enumerate}[label=(\alph*)]
\item \label{item:compact} $r$ is continuous and positive over $\rset$,
\item \label{item:tail:q} there exists $\alpha>2$ such that
  $r(u)|u|^{\alpha+1}$ is bounded away from $\infty$ as $|u|\to\infty$ and  from 0
  as $u\to\infty$,
\end{enumerate}
\end{hyp}
\begin{hyp}{E}
\item \label{item:g} the density $h$ satisfies:
\begin{enumerate}[label=(\alph*)]
\item $h$ is continuous and positive over $\rset$, and  $\lim_{|v| \to \infty} h(v)=0$,
\item \label{item:g:moment:mino} there exist $\beta \in [1,\alpha-1)$ (where
  $\alpha$ is given in \ref{item:q}) and $b,c>0$ such that
  $\PE(|V_0|^\beta)<\infty$ and $h(v)\geq b\,\rme^{-c|v|^{\beta}}$ for all
  $v\in\rset$.
\end{enumerate}
\end{hyp}
For example, a symmetric Pareto distribution with a parameter strictly larger
than 2 satisfies \ref{item:q} and provided that $\alpha>3$, \ref{item:g} holds
with a centered Gaussian distribution.  The model is parameterized by
$\theta=(m,a) \in \Theta\eqdef [\underline{m},\overline{m}] \times
[\underline{a},\overline{a}]$ where $0<\underline{m}<\overline{m}$ and
$\underline{a}<\overline{a}$. In this model, the Markov transition $Q^\theta$
of $\sequence{X}[k][\zsetp]$ has a transition density $q^\theta$ with respect to
the dominating measure $\mu(\rmd x)=\lambda(\rmd x) +\delta_0(\rmd x)$, which
can be written as follows: for all $(x,x') \in \rsetp^2$,
\begin{equation}
\label{eq:expression:q}
q^\theta(x;x')= r(x'-x+m)\1\{x'>0\}+ \left(\int_{-\infty}^{m-x} r(u) \;\rmd u\right)\1\{x'=0\} \eqsp.
\end{equation}
Moreover, \eqref{eq:def:exemple:hmm} implies
\begin{equation}
\label{eq:expression:g}
g^\theta(x;y)=h(y-a x)\eqsp.
\end{equation}

Following  \cite{jarner:roberts:2002}, we have the following lemma.
\begin{lemma} \label{lem:exemple:hmm}
 Assume \ref{item:q} and  \ref{item:g}. For all $\theta \in \Theta$, the Markov kernel $Q^\theta$ is not geometrically ergodic. Moreover, $Q^\theta$ is polynomially ergodic and its (unique) stationary distribution  $\pi_1^\theta$, defined on $\Xset=\rsetp$, satisfies $\int \pi_1^\theta(\rmd x) x^{\beta}<\infty$, for all $\beta \in [1,\alpha-1)$.
\end{lemma}
\begin{proof}
The proof of this Lemma is postponed to \autoref{sec:app:example:hmm} in
\autoref{sec:postponed-proofs}.
\end{proof}

\begin{proposition} \label{prop:consistanceYexample} Consider the HMM of
  \autoref{exple:hmm} under Assumptions~\ref{item:q} and \ref{item:g}. Then
  \ref{assum:gen:identif:unique:pi} holds and we define $\mathbb{P}^{\theta}$,
  $\tilde{\mathbb{P}}^{\theta}$ and the equivalence class $[\theta]$ as in
  \autoref{def:equi:theta}.  Moreover, for any probability measure $\Xinit$,
  the MLE $\mlY{\Xinit,n}$ is equivalence-class consistent, that is, for any
  $\thv\in\Theta$,
\begin{equation*}
\lim_{n \to \infty} \met(\mlY{\Xinit,n},[\thv])=0,\quad \tilde \PP^\thv\as
\end{equation*}
\end{proposition}
\begin{proof}
To apply \autoref{thm:consistanceY}, we need to check
\ref{assum:gen:identif:unique:pi} and
\ref{assum:q:positive}--\ref{assum:continuity}. First observe that Assumption
\ref{assum:gen:identif:unique:pi} immediately follows from \autoref{rem:A1-hmm}
and \autoref{lem:exemple:hmm}, and Assumptions \ref{assum:q:positive} and \ref{assum:g-unif-bounded} directly follow from the
positiveness of the density $r$ and the
boundedness of the density $h$, respectively. Now, using \ref{item:q}-\ref{item:compact}, it
can be easily shown that all compact sets are local Doeblin sets and this in
turn implies, via $\lim_{|x| \to \infty} h(x)=0$, that Assumption
\ref{assum:fullyDominated:filter}-\ref{item:condition-L-K} is satisfied. We now
check \ref{assum:fullyDominated:filter}-\ref{item:mino-g-simple}. By
\ref{item:q}-\ref{item:compact}, we have for all compact sets $D$,
$\inf\set{r(x'-x+m)}{(x,x',m) \in D^2 \times [\underline{m},\overline{m}]} >0$,
which by \eqref{eq:expression:q} implies
$$
\inf_{\theta \in \Theta}\inf_{x \in D} Q^\theta(x;D) > 0\eqsp.
$$
To obtain \ref{assum:fullyDominated:filter}-\ref{item:mino-g-simple}, it thus remains to show
$$
\tilde\PE^\thv\left[\ln^- \inf_{\theta \in \Theta} \inf_{x \in D} g^\theta(x; Y_0)\right] < \infty\eqsp.
$$
By \ref{item:g}-\ref{item:g:moment:mino}, there exist positive constants $b$ and $c$ such that $h(v)\geq b \rme^{-c |v|^\beta}$. Plugging this into \eqref{eq:expression:g} yields
\begin{multline*}
\tilde\PE^\thv\left[\ln^- \inf_{\theta \in \Theta} \inf_{x \in D} g^\theta(x; Y_0)\right] \leq \tilde\PE^\thv\left[ |\ln b|+c(|Y_0|+\overline{a} \sup_{x \in D}|x|)^\beta\right]\\
=\PE^\thv\left[ |\ln b|+c(|aX_0+V_0|+\overline{a} \sup_{x \in D}|x|)^\beta\right]<\infty\eqsp,
\end{multline*}
where the finiteness follows from \ref{item:g}-\ref{item:g:moment:mino} and
\autoref{lem:exemple:hmm}. Finally, \ref{assum:fullyDominated:filter} is
satisfied.  \ref{assum:majo-g} is checked by writing
$$
\displaystyle\tilde \PE^\thv\left[ \lnp\sup_{\theta \in \Theta} \sup_{x \in \Xset} g^\theta(x;Y_0)\right] \leq \lnp \sup_{x \in \rset} h(x)<\infty\eqsp.
$$

To obtain \ref{assum:continuity}, we show by induction on $n$ that for all $n\geq1$, $\chunk{y}{0}{n-1} \in \rset^n$ and $x_0 \in \rsetp$, the function $\theta \mapsto \Pblock[\theta]{\chunk{y}{0}{n-1}}(x_0;\Xset)$ is continuous on $\Theta$. The case where $n=1$ is obvious since $\Pblock[\theta]{y_0}(x_0;\Xset)=g^\theta(x_0;y_0)=h(y_0-ax_0)$. We next assume the induction hypothesis on $n$ and note that
$$
\Pblock[\theta]{\chunk{y}{0}{n}}(x_0;\Xset)=g^\theta(x_0;y_0)\int \mu(\rmd x_1) q^\theta(x_0;x_1) \Pblock[\theta]{\chunk{y}{1}{n}}(x_1;\Xset)\eqsp.
$$
The continuity of $\theta \mapsto g^\theta(x_0;y_0)$ follows from
\eqref{eq:expression:g} and the continuity of $h$. Similarly, the continuity of
$\theta \mapsto q^\theta(x_0;x_1)$ follows from \eqref{eq:expression:q} and the
continuity of $r$. Moreover, $\theta \mapsto
\Pblock[\theta]{\chunk{y}{1}{n}}(x_1;\Xset)$ is continuous by the induction
assumption. The continuity of $\theta \mapsto \int \mu(\rmd x_1)
q^\theta(x_0;x_1) \Pblock[\theta]{\chunk{y}{1}{n}}(x_1;\Xset)$ then follows
from the Lebesgue convergence theorem provided that
\begin{equation}
\label{eq:hmm:exemple:last}
\int \mu(\rmd x_1) \sup_{\theta \in \Theta} q^\theta(x_0;x_1) \Pblock[\theta]{\chunk{y}{1}{n}}(x_1;\Xset)<\infty
\end{equation}
holds. Note further that by the expression of $q^\theta(x_0;x_1)$ given in
\eqref{eq:expression:q} and the tail assumption \ref{item:q}-\ref{item:tail:q},
we obtain for all $x_0\in\Xset$,
$$
\int \mu(\rmd x_1) \sup_{\theta \in \Theta} q^\theta(x_0;x_1)<\infty\eqsp.
$$
Combining with that $\Pblock[\theta]{\chunk{y}{1}{n}}(x_1;\Xset) \leq (\sup_{x \in \rset} h(x))^n$ yields \eqref{eq:hmm:exemple:last}. Finally, we have \ref{assum:continuity}, and thus \autoref{thm:consistanceY} holds under \ref{item:q} and \ref{item:g}.
\end{proof}

\section{Application to observation-driven models}\label{sec:observ-driv-model}

Observation-driven models are a subclass of partially dominated and
partially observed Markov models.

We split our study of the observation-driven model into several parts. Specific
definitions and notation are introduced in
\autoref{sec:definitions-notation-od}. Then we provide sufficient conditions
that allow to apply our general result \autoref{thm:mainresult}, that is,
$\Theta_\star=[\thv]$.  This is done in \autoref{sec:identifiability-ODM}.
\subsection{Definitions and notation}
\label{sec:definitions-notation-od}
Observation-driven models are formally defined as follows.
\begin{definition}
  Consider a partially observed and partially dominated Markov model given in \autoref{def:partially-dom} with Markov
  kernels $\nsequence{K^\theta}{\theta\in\Theta}$. We say that this model is an observation-driven model if the kernel $K^\theta$ satisfies
\begin{equation}\label{eq:def:observation-driven}
K^\theta((x,y);\rmd x'\rmd y') =  \delta_{\psi^\theta_{y}(x)}(\rmd x') \;
G^\theta (x';\rmd y') \eqsp,
\end{equation}
where $\delta_a$ denotes the Dirac mass at point $a$, $G^\theta$ is a
probability kernel on $\Xset \times \Ysigma$ and $\nsequence{(x,y) \mapsto
\psi^\theta_y(x)}{\theta \in \Theta}$ is a family of measurable functions
from $(\Xset\times \Yset, \Xsigma \otimes \Ysigma)$ to $(\Xset, \Xsigma)$.
Moreover, in this context, we always assume that $(\Xset,\Xmet)$ is a complete
separable metric space and $\Xsigma$ denotes the associated Borel
$\sigma$-field.
\end{definition}
Note that a  Markov chain $\dsequence{X}{Y}[k][\zsetp]$ with probability kernel given
by~\eqref{eq:def:observation-driven} can be equivalently defined by the following recursions
\begin{align}
\begin{split}
& X_{k+1}=\psi^\theta_{Y_{k}}(X_k) \eqsp, \label{eq:def:XY:theta:cl} \\
& Y_{k+1}| \chunk{X}{0}{k+1},\chunk{Y}{0}{k}\sim G^\theta (X_{k+1};\cdot) \eqsp.
\end{split}
\end{align}
The most celebrated example is the GARCH$(1,1)$ process, where
$G^\theta(x;\cdot)$ is a centered (say Gaussian) distribution with variance $x$ and
$\psi^\theta_y(x)$ is an affine function of $x$ and $y^2$.

As a special case of \autoref{def:partially-dom}, for all
$x \in \Xset$, $G^\theta (x;\cdot)$ is dominated by some $\sigma$-finite
measure $\nu$ on $(\Yset,\Ysigma)$ and we denote by $g^\theta (x;\cdot)$ its Radon-Nikodym
derivative, $g^\theta (x;y)=\frac{\rmd G^\theta (x;\cdot)}{\rmd
  \nu}(y)$. A dominated parametric observation-driven model is thus defined by the
collection $\nsequence{(g^\theta,\psi^\theta)}{\theta\in\Theta}$. Moreover,~(\ref{eq:g-evrywhere-strictly-positive}) may be rewritten in this case: for all $(x,y) \in \Xset \times \Yset$ and for all $\theta\in\Theta$,
\begin{equation*}
g^\theta(x;y)>0\;.
\end{equation*}
Under \ref{assum:gen:identif:unique:pi}, we assume that the model is well-specified, i.e., the
observation sample $(Y_1,\ldots,Y_n)$ is distributed according to
$\tilde\PP^{\thv}$ for some unknown parameter $\thv$. The inference of $\thv$ is based on
the conditional likelihood of $(Y_1,\ldots,Y_n)$ given $X_1=x$ for an
arbitrary $x\in\Xset$. The corresponding density function with respect to
$\nu^{\otimes n}$ is, under parameter $\theta$,
\begin{equation} \label{eq:lkd:Y:cl:X1=x}
\chunk{y}{1}{n} \mapsto \prod_{k=1}^{n} g^\theta\left(\f{\chunk{y}{1}{k-1}}(x);y_{k}\right) \eqsp,
\end{equation}
where, for any vector $\chunk{y}{1}{p}=(y_1,\dots,y_p)\in\Yset^p$,
$\f{\chunk{y}{1}{p}}$ is the $\Xset\to\Xset$ function defined as the successive
composition of  $\psi^\theta_{y_1}$,  $\psi^\theta_{y_2}$, ..., and $\psi^\theta_{y_p}$,
\begin{equation}
\label{eq:notationItere:f:cl}
\f{\chunk{y}{1}{p}}=\psi^\theta_{y_p} \circ \psi^\theta_{y_{p-1}} \circ \dots \circ \psi^\theta_{y_1}\,,
\end{equation}
with the convention $\f{\chunk{y}{s}{t}}(x)=x$ for $s>t$. Then the corresponding
(conditional) MLE $\mlY{x,n}$ of the parameter $\theta$ is defined by
\begin{equation}
\label{eq:defi:mle}
\mlY{x,n} \in \argmax_{\theta \in \Theta} \lkdM[x,n]{\chunk{Y}{1}{n}}[\theta]\eqsp,
\end{equation}
where
\begin{equation}
\label{eq:defi:lkdM}
\lkdM[x,n]{\chunk{y}{1}{n}}[\theta]\eqdef   n^{-1}\ln \left( \prod_{k=1}^{n} g^\theta\left(\f{\chunk{y}{1}{k-1}} (x);y_k\right) \right)\eqsp.
\end{equation}

We will provide simple conditions for the consistency of $\mlY{x,n}$ in the
sense that, with probability tending to one, for a well chosen $x$,
$\mlY{x,n}$ belongs to a neighborhood of the equivalence class $[\thv]$ of $\thv$, as
given by \autoref{def:equi:theta}.
\subsection{Identifiability}
\label{sec:identifiability-ODM}
Let us consider the following assumptions.
\begin{hyp}{C}
\item \label{assum:exist:x0:O} For all $\theta\neq\thv\in\Theta$, there exist $x \in\Xset$ and a
  measurable function $\f[\theta,\thv]{\cdot}$ defined on
  $\Yset^{\zsetn}$ such that
\begin{equation}
  \label{eq::exist:x0:O}
\lim_{m \to \infty}\f{\chunk{Y}{-m}{0}}(x)=\f[\theta,\thv]{\chunk{Y}{-\infty}{0}}, \quad \tilde \PP^{\thv}\as
\end{equation}
\item \label{assum:g-cont:O} For all $\theta\in\Theta$ and $y\in\Yset$, the function
  $x\mapsto g^\theta(x; y)$ is continuous on $\Xset$.
\item \label{assum:phi-cont:O} For all $\theta\in\Theta$ and $y\in\Yset$, the function $x\mapsto\psi^\theta_y(x)$ is continuous on $\Xset$.
\end{hyp}
In observation-driven models, the kernel $\kappa^\theta$ defined in~\eqref{eq:kappa} reads
\begin{align}\nonumber
\kap{y, y'}{x; \rmd x'}&=g^\theta(x';y')\;
\delta_{\psi^\theta_{y}(x)}(\rmd x')\\
  \label{eq:kap:O}
&=g^\theta\left(\psi^\theta_{y}(x);y'\right)\; \delta_{\psi^\theta_{y}(x)}(\rmd x')\eqsp,
\end{align}
and the probability kernel $\Phi_{x,n}^\theta$ in \autoref{def:Phi:x:f} reads, for all $x\in\Xset$ and $\chunk{y}{0}{n}\in\Yset^{n+1}$,
\begin{equation}
  \label{eq:Phi:O}
\Phi_{x,n}^\theta(\chunk{y}{0}{n}; \cdot)=\delta_{\f{\chunk{y}{0}{n-1}}(x)}
\end{equation}
(the Dirac point mass at $\f{\chunk{y}{0}{n-1}}(x)$).
Using these expressions, we get the following result which is a special case of
\autoref{thm:mainresult}.
\begin{theorem}\label{thm:main:result:obs}
Assume that~\ref{assum:gen:identif:unique:pi} holds in the observation-driven
model setting and define
$\mathbb{P}^{\theta}$, $\tilde{\mathbb{P}}^{\theta}$ and $[\theta]$ as in
\autoref{def:equi:theta}. Suppose that
Assumptions~\ref{assum:exist:x0:O},~\ref{assum:g-cont:O} and~\ref{assum:phi-cont:O} hold and define
    $p^{\theta,\thv}(\cdot|\cdot)$ by setting, for
  $\tilde\PP^{\thv}$-a.e. $\chunk{y}{-\infty}{0}\in\Yset^{\zsetn}$,
  \begin{equation}
    \label{eq:def-p1-od}
p^{\theta,\thv}(y_1\,|\,\chunk{y}{-\infty}{0}) =
\begin{cases}
  g^\theta\left(\f[\theta,\thv]{\chunk{y}{-\infty}{0}};y_1\right)&\text{ if
    $\theta\neq\thv$,}\\
p_1^\theta(y_1\,|\,\chunk{y}{-\infty}{0})\text{ as defined by~(\ref{eq:p-theta-density})}&\text{ otherwise.}
\end{cases}
  \end{equation}
Then, for all $\thv\in\Theta$, we have
\begin{equation}\label{eq:ThetaStar-included-class-od}
\argmax_{\theta\in\Theta}\tilde{\PE}^{\thv}\left[\ln
  p^{\theta,\thv}(Y_1|\chunk{Y}{-\infty}{0})\right] = [\thv]\;.
\end{equation}
\end{theorem}

\begin{proof}
We apply \autoref{thm:mainresult}.  It is thus sufficent to show that
\ref{assum:exist:x0:O},~\ref{assum:g-cont:O} and~\ref{assum:phi-cont:O}
  implies~\ref{assum:exist:phi:theta} with
\begin{equation}\label{eq:def-Phi-od}
\Phi^{\theta,\thv}(\chunk{y}{-\infty}{0};\cdot)=\delta_{\f[\theta,\thv]{\chunk{y}{-\infty}{-1}}},\quad
\text{for all $\chunk{y}{-\infty}{0}\in\Yset^{\zsetn}$}\;,
\end{equation}
and that for $\theta\neq\thv$,  the conditional density $p^{\theta,\thv}$ defined by~(\ref{eq:def-p_1}) satisfies
\begin{equation}
\label{eq:def-p_1-od}
p^{\theta,\thv}(y|\chunk{Y}{-\infty}{0}) = g^\theta\left(\f[\theta,\thv]{\chunk{Y}{-\infty}{0}};y\right)\eqsp,
\quad\tilde\PP^{\thv}\as
\end{equation}

By \autoref{lem:help:I2},
it is sufficient to prove that Assumption~\ref{assum:help:exist:phi} holds for
the kernel $\Phi^{\theta,\thv}$ defined above.  Denote by
$\mathcal{C}(\Xset)$  the set
of continuous functions on $\Xset$, and by $\mathcal{C}_b(\Xset)$ the set of
bounded functions in $\mathcal{C}(\Xset)$. By~\cite[Theorem~6.6,
Chapter~6]{parthasarathy:2005}, there is a countable and separating subclass
$\mathcal{F}$ of nonnegative functions in $\mathcal{C}_b(\Xset)$ such that
$\1_\Xset\in\mathcal{F}$. Now, let us take $\theta,\thv\in\Theta$ and
$f\in\mathcal{F}$.  Then, by~\ref{assum:g-cont:O}, \ref{assum:phi-cont:O}
and~(\ref{eq:kap:O}), we
have
$$
\mathcal{F}_f^\theta = \left\{x \mapsto \kap{y, y'}{x;
    f}:\; (y,y')\in\Yset^2\right\}\subset\mathcal{C}(\Xset)\;.
$$
By~(\ref{eq:Phi:O}),~\ref{assum:exist:x0:O} and~(\ref{eq:def-Phi-od}), we
obtain~\ref{assum:help:exist:phi} with $x$ chosen as in~\ref{assum:exist:x0:O}.

To conclude, we need to show~(\ref{eq:def-p_1-od}).  Note that~(\ref{eq:def-Phi-od})
together with~(\ref{eq:kap:O}) and the usual definition~(\ref{eq:def-p_1}) of
$p^{\theta,\thv}$ yields
$$
p^{\theta,\thv}(y|\chunk{y}{-\infty}{0}) =g^\theta\left(\psi^\theta_{y_0}\left(\f[\theta,\thv]{\chunk{y}{-\infty}{-1}}\right);y\right)\;.
$$
By Assumption~\ref{assum:phi-cont:O} and the definition of
$\f[\theta,\thv]{\cdot}$ in~\ref{assum:exist:x0:O}, we get~(\ref{eq:def-p_1-od}).
\end{proof}

\subsection{Examples}
In the context of observation-driven time series, easy-to-check conditions are
derived in \cite{douc2015handy} in order to establish the convergence
of the MLE $\mlY{x,n}$ defined by~(\ref{eq:defi:mle}) to the maximizing set of
the asymptotic normalized log-likelihood. It turns out that the conditions of
\cite[Theorem~3]{douc2015handy} also imply the conditions of
\autoref{thm:main:result:obs}.  More precisely, the assumptions (B-2) and
(B-3) of \cite[Theorem~1]{douc2015handy} are stronger than
\ref{assum:g-cont:O} and \ref{assum:phi-cont:O} used in
\autoref{thm:main:result:obs} above, and it is shown that the assumptions of
\cite[Theorem~1]{douc2015handy} imply \ref{assum:exist:x0:O} (see the
proof of Lemma~2 in Section~6.3 of \cite{douc2015handy}). Moreover, the
conditions of Theorem~1 are shown to be satisfied in the context of
Examples~\ref{example:Nbigarch:defi} and~\ref{example:nmgarch:defi} (see
\cite[Theorem~3 and Theorem~4]{douc2015handy}), provided that $\Theta$
in~(\ref{eq:defi:mle}) is a compact metric space such that
\begin{enumerate}
\item in the case of \autoref{example:Nbigarch:defi}, all $\theta=(\omega, a, b,r)\in\Theta$ satisfy $rb+a<1$;
\item in the case of \autoref{example:nmgarch:defi}, all $\theta=\left(\boldsymbol{\gamma}, \boldsymbol{\omega}, \mathbf{A},
  \mathbf{b}\right)\in\Theta$ are such that the spectral radius of
$\mathbf{A}+\mathbf{b}\boldsymbol{\gamma}^T$ is strictly less than 1.
\end{enumerate}
Under these assumptions, we conclude that the MLE is equivalence-class
consistent for both examples, which up to our best knowledge had not been proven so far.

\begin{appendix}

\section{Postponed proofs}
\label{sec:postponed-proofs}

\subsection{Proof of Eq.~\eqref{eq:phi:theta:theta}}
\label{sec:proof-eq10}
Let $\theta\in\Theta$. Recall that in \autoref{rem:intuition:kernel},
$\Phi^{\theta}$  is defined as the probability kernel of the conditional
distribution of $X_0$ given $\chunk{Y}{-\infty}{0}$ under $\PP^\theta$, that
is, for all $A\in\Xsigma$,
\begin{equation*}
\Phi^{\theta}(\chunk{Y}{-\infty}{0};A)= \PP^{\theta}\left(X_0\in A\left.\right|\chunk{Y}{-\infty}{0}\right), \quad \tilde\PP^\theta\as
\end{equation*}
Conditioning on $X_0,Y_0$ and using the definition of $\kappa^\theta$
in~(\ref{eq:kappa}), we get that, for all $A\in\Xsigma, B\in\Ysigma$ and
$C\in\Ysigma^{\otimes{\zsetn}}$,
\begin{align}\nonumber
&\PP^\theta\left(X_1\in A, Y_1\in B, \chunk{Y}{-\infty}{0}\in C\right) \\\nonumber
&=\PE^\theta\left[\int_B\kap{Y_0,y_1}{X_0;A} \1_{C}(\chunk{Y}{-\infty}{0})\;\nu(\rmd y_1)\right] \\
\label{eq:app:prob:phi}
&=\tilde\PE^\theta\left[\int_{\Xset\times B} \Phi^{\theta}(\chunk{Y}{-\infty}{0};\rmd x_0)\kap{Y_0,y_1}{x_0;A} \1_{C}(\chunk{Y}{-\infty}{0})\;\nu(\rmd y_1)\right]\eqsp.
\end{align}
Let us denote
$$
\hat\Phi^{\theta}(\chunk{Y}{-\infty}{0},y_1;A)
=\frac{\int \Phi^{\theta}(\chunk{Y}{-\infty}{0};\rmd x_0)\kap{Y_0,y_1}{x_0;A}}{\int \Phi^{\theta}(\chunk{Y}{-\infty}{0};\rmd x_0)\kap{Y_0,y_1}{x_0;\Xset}}\;,
$$
which is always defined since the denominator does not vanish
by~\autoref{rem:kappa-is-positive}. With this notation, we deduce
from~(\ref{eq:app:prob:phi}) that
\begin{multline*}%\label{eq:app:prob:phi}
\PP^\theta\left(X_1\in A, Y_1\in B, \chunk{Y}{-\infty}{0}\in C\right)\\
=\tilde\PE^\theta\Bigg[\int_B \hat\Phi^{\theta}(\chunk{Y}{-\infty}{0},y_1;A)
\left(\int \Phi^{\theta}(\chunk{Y}{-\infty}{0};\rmd x_0)\kap{Y_0,y_1}{x_0;\Xset}\right)\\
\1_{C}(\chunk{Y}{-\infty}{0})\;\nu(\rmd y_1)\Bigg]\eqsp.
\end{multline*}
This can be more compactly written as
\begin{multline}\label{eq:joint-proba-X1-Y0infty}
\PP^\theta\left(X_1\in A, Y_1\in B, \chunk{Y}{-\infty}{0}\in C\right)
=\tilde\PE^\theta\Bigg[\int \hat\Phi^{\theta}(\chunk{Y}{-\infty}{0},y_1;A)\\
\1_B(y_1)\,\1_{C}(\chunk{Y}{-\infty}{0})
\Phi^{\theta}(\chunk{Y}{-\infty}{0};\rmd x_0)\kap{Y_0,y_1}{x_0;\Xset}
\;\nu(\rmd y_1)\Bigg]\eqsp.
\end{multline}
Observe that~(\ref{eq:app:prob:phi}) with $A=\Xset$ provides
a way to write $\tilde\PE^\theta\left[g(\chunk{Y}{-\infty}{0},Y_1)\right]$ for
$g=\1_{C\times B}$ that can be extended to any nonnegative measurable function $g$ defined on
$\Yset^{\zsetn}\times\Yset$ as
\begin{multline*}
\tilde\PE^\theta\left[g(\chunk{Y}{-\infty}{0},Y_1)\right] \\
=\tilde\PE^\theta\left[\int
g(\chunk{Y}{-\infty}{0},y_1)  \Phi^{\theta}(\chunk{Y}{-\infty}{0};\rmd x_0)\kap{Y_0,y_1}{x_0;\Xset}
\nu(\rmd y_1)\right]\eqsp.
\end{multline*}
Now, we observe that the right-hand side of~(\ref{eq:joint-proba-X1-Y0infty})
can be interpreted as the right-hand side of the previous display with
$g(\chunk{Y}{-\infty}{0},y_1)=\hat\Phi^{\theta}(\chunk{Y}{-\infty}{0},y_1;A)\1_B(y_1)\1_{C}(\chunk{Y}{-\infty}{0})$.
Hence, we conclude that,
for all $A\in\Xsigma$ and $C\in\Ysigma^{\otimes{\zsetn}}$,
$$
\PP^\theta\left(X_1\in A, Y_1\in B, \chunk{Y}{-\infty}{0}\in C\right)
=\PE^\theta\Bigg[\hat\Phi^{\theta}(\chunk{Y}{-\infty}{0},Y_1;A)
\1_B(Y_1)\1_{C}(\chunk{Y}{-\infty}{0})\Bigg]\eqsp.
$$
Notice that
$\hat\Phi^{\theta}(\chunk{Y}{-\infty}{0},Y_1;A)$ precisely is the probability
kernel on $\left(\Yset^{\zsetn}\times\Yset\right)\times\Xsigma$ appearing on the
left-hand side of~(\ref{eq:phi:theta:theta}). The last display implies that
this probability kernel is the conditional distribution of $X_1$ given
$\chunk{Y}{-\infty}{1}$ under $\PP^\theta$, which concludes the proof of~(\ref{eq:phi:theta:theta}).
\subsection{Proof of~\autoref{lem:imp:p;seq}}
\label{sec:proof-lem:imp}
First observe that,  by induction on $n$, having ~\eqref{eq:sub:lem:p}
for all $n\geq2$ is equivalent to having, for all $n\geq2$,
\begin{multline*}
p^{\theta,\thv} (\chunk{Y}{1}{n}|\chunk{Y}{-\infty}{0}) \\
=p^{\theta,\thv}(Y_{n}|\chunk{Y}{-\infty}{n-1})
p^{\theta,\thv}(Y_{n-1}|Y_{-\infty:n-2})\cdots p^{\theta,\thv} (Y_{1}|\chunk{Y}{-\infty}{0}),\ \tilde{\PP}^{\thv}\as\eqsp,
\end{multline*}
which, using that $\tilde{\PP}^{\thv}$ is shift-invariant,
is in turn equivalent to having that, for all $n\geq2$,
\begin{equation}\label{eq:sub:lem:p2}
p^{\theta,\thv} (\chunk{Y}{1}{n}|\chunk{Y}{-\infty}{0})
=p^{\theta,\thv} (\chunk{Y}{2}{n}|\chunk{Y}{-\infty}{1})p^{\theta,\thv} (Y_{1}|\chunk{Y}{-\infty}{0}), \quad \tilde{\PP}^{\thv}\as
\end{equation}
Thus to conclude the proof, we only need to show that~(\ref{eq:sub:lem:p2})
holds for all $n\geq2$. By \autoref{def:cond-densities-given-past}, we have, for all
$n\geq2$ and $y_{-\infty:n}\in\Yset^{\zset_-}$,
\begin{multline*}
p^{\theta,\thv} (\chunk{y}{2}{n}|\chunk{y}{-\infty}{1})p^{\theta,\thv} (y_{1}|\chunk{y}{-\infty}{0})\\
=\int\Phi^{\theta,\thv}(\chunk{y}{-\infty}{1};\rmd x_1)p^{\theta,\thv} (y_{1}|\chunk{y}{-\infty}{0})\prod_{k=1}^{n-1} \kap{y_{k}, y_{k+1}}{x_{k}; \rmd x_{k+1}}\eqsp.
\end{multline*}
Using~\ref{assum:exist:phi:theta} we now get, for all $n\geq2$,
\begin{multline*}
p^{\theta,\thv} (\chunk{Y}{2}{n}|\chunk{Y}{-\infty}{1})p^{\theta,\thv}
(Y_{1}|\chunk{Y}{-\infty}{0})\\
=\int\Phi^{\theta,\thv}(\chunk{Y}{-\infty}{0};\rmd x_0)\prod_{k=0}^{n-1}
\kap{Y_{k}, Y_{k+1}}{x_{k}; \rmd x_{k+1}},\quad  \tilde{\PP}^{\thv}\as
\end{multline*}
We conclude ~(\ref{eq:sub:lem:p2}) by observing that, according to
\autoref{def:cond-densities-given-past}, the second line of the last display is
$p^{\theta,\thv} (\chunk{Y}{1}{n}|\chunk{Y}{-\infty}{0})$.
\subsection{Proof of~\autoref{lem:exemple:hmm}}
\label{sec:app:example:hmm}
Let $\beta \in [1,\alpha-1)$. Since $1+\beta<\alpha$ and by \ref{item:q}-\ref{item:tail:q}, we obtain
$\PE\left[(U_0^+)^{1+\beta}\right]<\infty$. Combining this with $\PE(U_0-m)=-m<0$, we may apply
\cite[Proposition 5.1]{jarner:roberts:2002} so that the Markov kernel
$Q^\theta$ is polynomially ergodic and thus admits a unique stationary
distribution $\pi_1^\theta$, which is well defined on $\Xset=\rsetp$. Moreover,
\cite[Proposition 5.1]{jarner:roberts:2002} also shows that there exist a
finite interval $C = [0,x_0]$ and some constants $\varrho, \varrho' \in(0,\infty)$ such that
$$
Q^\theta V\leq V-\varrho W +\varrho' \1_C\eqsp,
$$
where $V(x) = (1+x)^{1+\beta}$ and $W(x)=(1+x)^{\beta}$. Applying \cite[Theorem 14.0.1]{meyn:tweedie:1993} yields
$$\int \pi_1^\theta(\rmd x) x^{\beta} \leq \pi_1^\theta W<\infty\eqsp.$$
It remains to show that the kernel $Q^\theta$ is not geometrically ergodic for all $\theta \in \Theta$ and this will be done by contradiction.

Now suppose on the contrary that the kernel $Q^\theta$ is
geometrically ergodic for some $\theta \in \Theta$. Since the singleton
$\{0\}$ is an accessible atom (for $Q^\theta$),  then there exists some $\rho>1$ such that
$$\sum_{k=0}^\infty \rho^k \left|(Q^\theta)^k(0,\{0\})-\pi_1^\theta(\{0\})\right|<\infty\eqsp.$$
Hence, the atom $\{0\}$ is geometrically ergodic as defined in \cite[Section
15.1.3]{meyn:tweedie:1993}. Applying \cite[Theorem 15.1.5]{meyn:tweedie:1993},
then there exists some $\kappa>1$  such that $\PE_{0}[\kappa^{\tau_0}]<\infty$, where
$\tau_0=\inf \{n \geq 1:\; X_n=0\}$ is the first return time to $\{0\}$.

Recall that the i.i.d. sequence \sequence{U}[k][\zsetp] is linked to \sequence{X}[k][\zsetp] through \eqref{eq:def:exemple:hmm}, and note that $\PE_{0}[\kappa^{\tau_0}]=\PE[\kappa^{\tau(0)}]$, where we have set for all $u \in \rset$,
\begin{align*}
&\tau(u) \eqdef \inf\left\{n \geq 1:\; \sum_{k=1}^n (U_k-m)<u\right\} \eqsp.
\end{align*}
Now, denote
\begin{align*}
&\tilde \tau(u) \eqdef \inf\left\{n \geq 1:\; \sum_{k=1}^n (U_{k+1}-m)<u\right\}\eqsp.
\end{align*}
To arrive at the contradiction, it is finally sufficient to show that for all $\kappa>1$, $\PE[\kappa^{\tau(0)}]=\infty$. Actually, we will show that there exists a constant $\gamma >0$ such that
\begin{equation}
\label{eq:hmm:exemple:one}
\liminf_{u \to \infty}\kappa^{-\gamma u}\PE[\kappa^{\tau(-u+m)}] >0 \eqsp.
\end{equation}
This will indeed imply $\PE[\kappa^{\tau(0)}]=\infty$ by writing
\begin{align}
\PE[\kappa^{ \tau(0)}]&\geq \PE[\kappa^{\tau(0)} \1\{U_1 \geq m\}] = \PE[\kappa^{1+\tilde \tau(-U_1+m)} \1\{U_1 \geq m\}] \nonumber \\
&=\PE\left[\int_m^\infty \kappa^{1+\tilde \tau(-u+m)}  r(u) \rmd u\right]=\kappa \int_m^\infty  \PE[\kappa^{\tau(-u+m)}]  r(u) \rmd u  \eqsp, \label{eq:bidule}
\end{align}
where the last equality follows from $\tau\overset{d}{=} \tilde\tau$.  Provided
that \eqref{eq:hmm:exemple:one} holds, the right-hand side of \eqref{eq:bidule} is
infinite since $r(u)\gtrsim u^{-\alpha-1}$ as $u \to \infty$
by~\ref{item:q}-\ref{item:tail:q}.

We now turn to the proof of \eqref{eq:hmm:exemple:one}. By Markov's inequality, we have for any $\gamma>0$,
\begin{equation} \label{eq:hmm:exemple:zero}
\kappa^{-\gamma u}\PE[\kappa^{\tau(-u+m)}] \geq \PP(\tau(-u+m)>\gamma u)  \eqsp.
\end{equation}
Now, let $M_n=\sum_{k=1}^n U_{i}$, $n \geq 1$, and note that for all nonnegative $u$,
\begin{align}
\left\{\left(\inf_{1\leq k\leq \gamma u} M_k \right)-\gamma u m \geq -u+m\right\} &\subset \left\{\inf_{1\leq k\leq \gamma u} (M_k -k m)  \geq -u+m\right\} \nonumber\\
 &= \{\tau(-u+m)> \gamma u\} \eqsp. \label{eq:hmm:exemple:two}
\end{align}
Moreover, since \sequence{U}[k][\zsetpnz] is i.i.d. and centered, Doob's maximal inequality implies,  for all  $\tilde \gamma>0$,
\begin{align}
\PP\left(\inf_{1\leq k\leq \gamma u} M_k  < -\tilde \gamma\right)&\leq\PP\left(\sup_{1\leq k\leq \gamma u} |M_k|  > \tilde \gamma\right) \nonumber\\
&\leq \frac{\PE\left[|M_{\lfloor \gamma u\rfloor}|\right]}{\tilde \gamma}\leq \frac{\lfloor \gamma u\rfloor \PE\left[|U_1|\right]}{\tilde \gamma} \eqsp. \label{eq:hmm:exemple:three}
\end{align}
Now, pick $\gamma>0$ sufficiently small so that $\gamma \PE\left[|U_1|\right]/(1-\gamma
m)<1$. Observe that for this $\gamma$, $\tilde \gamma=(1-\gamma m)u-m$ is positive
for $u$ sufficiently large, so that combining \eqref{eq:hmm:exemple:three} with
\eqref{eq:hmm:exemple:two} and \eqref{eq:hmm:exemple:zero} yields
$$
\liminf_{u \to \infty} \kappa^{-\gamma u}\PE[\kappa^{\tau(-u+m)}] \geq 1-\limsup_{u \to \infty}\frac{\lfloor \gamma u\rfloor \PE\left[|U_1|\right]}{(1-\gamma m)u-m}= 1-\frac{\gamma \PE\left[|U_1|\right]}{1-\gamma m}>0\eqsp.
$$
This shows \eqref{eq:hmm:exemple:one} and the proof is completed.
\end{appendix}

\section*{Acknowledgements}
The authors are thankful to the anonymous referee for insightful comments and helpful suggestions that led to improving this paper.

%\section*{Bibliography}
%\bibliographystyle{imsbib}
%\bibliographystyle{elsarticle-harv}
\bibliographystyle{amsxport}
\bibliography{monybibs}

\end{document}